\documentclass[11pt]{amsart}
\usepackage{amsfonts}
\usepackage{amsmath}
\usepackage{amsthm}
\usepackage{bm,bbm}
\usepackage[dvips]{graphicx}
\usepackage{caption}
\usepackage{color}

\usepackage{bigints}

\numberwithin{equation}{section}

\allowdisplaybreaks

\theoremstyle{plain}
\newtheorem{theorem}{Theorem}[section]
\newtheorem{lemma}[theorem]{Lemma}
\newtheorem{proposition}[theorem]{Proposition}

\theoremstyle{definition}
\newtheorem{definition}[theorem]{Definition}
\newtheorem{remark}[theorem]{Remark}

\def\beqn{\begin{equation}}
\def\beqn*{$$}
\def\eeqn{\end{equation}}

\newcommand{\bx}{{\bf x}}
\newcommand{\by}{{\bf y}}

\newcommand{\bi}{{\bf i}}
\newcommand{\bj}{{\bf j}}
\def\P{\mathbb{P}}

\def\E{\mathbb{E}}
\def\Pn{\mathcal P_n}

\newcommand{\reals}{{\mathbb R}}
\newcommand{\bbr}{\reals}

\newcommand{\bbn}{{\mathbb N}}

\newcommand{\X}{{\mathcal{X}}}
\newcommand{\Y}{{\mathcal{Y}}}
\newcommand{\I}{\mathcal I}

\newcommand{\one}{{\mathbbm 1}}

\newcommand{\remove}[1]{}

\newcommand{\C}{\check{C}}
\newcommand{\betakn}{\beta_{k,n}}
\newcommand{\Z}{\mathcal Z}

\newcommand{\hij}{h^{(i,j)}}

\newcommand{\DelL}{\Delta_L}
\newcommand{\gij}{g^{(i,j)}}
\newcommand{\genR}{R_{s,t,u,v}}
\newcommand{\genRi}{R_{s_i,t_i,u_i,v_i}}
\newcommand{\genH}{H_{s,t,u,v}}
\newcommand{\genHi}{H_{s_i,t_i,u_i,v_i}}

\newcommand{\A}{\mathcal A}
\newcommand{\M}{\mathcal M}
\newcommand{\seq}{n^{k+2}r_n^{d(k+1)}}
\newcommand{\Ik}{\mathcal I_{k+2}}
\newcommand{\supt}{\sup_{0 \le t \le T}}

\begin{document}

\bibliographystyle{abbrv}

%can adjust space between lines
\renewcommand{\baselinestretch}{1.05}

\title[Convergence of persistence diagram]
{Convergence of persistence diagram in the sparse regime}

\author{Takashi Owada}
\address{Department of Statistics\\
Purdue University \\
West Lafayette, 47907, USA}
\email{owada@purdue.edu}

\thanks{This research was partially supported by NSF grant DMS-1811428.}

\keywords{Stochastic topology, persistent homology, persistence diagram, persistent Betti number \vspace{.5ex}}

\begin{abstract}
The objective of this paper is to examine the asymptotic behavior of persistence diagrams associated with \v{C}ech filtration. A persistence diagram is a graphical descriptor of a topological and algebraic structure of geometric objects. We consider \v{C}ech filtration over a scaled random sample $r_n^{-1}\mathcal X_n = \{ r_n^{-1}X_1,\dots, r_n^{-1}X_n \}$, such that $r_n\to 0$ as $n\to\infty$. We treat persistence diagrams as a point process and establish their limit theorems in the sparse regime: $nr_n^d\to0$, $n\to\infty$. In this setting, we show that the asymptotics of the $k$th persistence diagram depends on the limit value of the sequence $n^{k+2}r_n^{d(k+1)}$. If $n^{k+2}r_n^{d(k+1)} \to \infty$, the scaled persistence diagram converges to a deterministic Radon measure almost surely in the vague metric. If $r_n$ decays faster so that $n^{k+2}r_n^{d(k+1)} \to c\in (0,\infty)$, the persistence diagram weakly converges to a limiting point process without normalization. Finally, if $n^{k+2}r_n^{d(k+1)} \to 0$, the sequence of probability distributions of a persistence diagram should be normalized, and the resulting convergence will be treated in terms of the $\mathcal M_0$-topology.
\end{abstract}
%In each of the limit theorems, the structure of a limit measure has also been clarified.

\maketitle

\section{Introduction}

The main theme of this paper is persistent homology and its diagrams associated with random geometric complexes. In applied topology, persistent homology is one of the tools most ubiquitously used to analyze data in a way robust to various deformations. In recent times, persistent homology has demonstrated its applicability in areas as diverse as sensor networks \cite{desilva:ghrist:2007}, bioinformatics \cite{dabaghian:memoli:frank:carlsson:2012}, computational chemistry \cite{martin:thompson:coutsias:watson:2010}, %manifold learning \cite{niyogi:smale:weinberger:2008},
linguistics \cite{port:gheorghita:guth:clark:liang:dasu:marcolli:2018}, astrophysics \cite{pranav:edelsbrunner:weygaert:vegter:kerber:jones:wintraecken:2017},  %pranav:adler:buchert:edelsbrunner:jones:schwartzman:wagner:weygaert:2019},
cancer genomics \cite{arsuaga:borrman:cavalcante:gonzalez:park:2015},
%camara:rosenbloom:emmett:levine:rabadan:2016},
and material science \cite{hiraoka:nakamura:hirata:escolar:matsueNishiura:2016}.  %murakami:kohara:kitamura:akola:inoue:hirata:hiraoka:onodera:obayashi:kalikka:hirao:musso:foster:idemoto:sakata:ohishi:2019}.

First, we present one illustrative example, which helps capture the essence of persistent homology.
A more rigorous description of persistent homology can be found in \cite{edelsbrunner:letscher:zomorodian:2002}, \cite{zomorodian:carlsson:2005}, and \cite{edelsbrunner:harer:2010}, while \cite{adler:bobrowski:borman:subag:weinberger:2010} and \cite{ghrist:2008} provide good introductory reading. In Figure \ref{f:persistence}, given a data set $\X=\{ x_1,\dots,x_n \}$, we wish to estimate an underlying topology of an annulus from the union of balls, $U(t):=\bigcup_{i=1}^n B(x_i,t)$, where $B(x,\rho)$ denotes a closed ball of radius $\rho$ centered at $x\in\bbr^d$. If the radius $t$ is suitably selected, as in Figure \ref{f:persistence}(b), $U(t)$ is very close in shape to an annulus. However, if $t$ is small, as in Figure \ref{f:persistence}(a), $U(t)$ simply consists of many small components and fails to recover the topology of an annulus. Similarly, if $t$ is too big, as in Figure \ref{f:persistence}(c), $U(t)$ fails to detect a hole at the center of an annulus. This example indicates that selecting an appropriate radius is not easy. To overcome this issue, persistent homology is aimed to detect a robust and topological structure of a manifold by tracking the creation and destruction of topological cycles, such as the closed loops in Figure \ref{f:persistence}.

Figure \ref{f:diagram-plot} visualizes the outcome of persistent homology for the annulus example in a two-dimensional plot known as a \emph{persistence diagram}. In this two-dimensional plot, the $x$-axis represents the time (radius) at which each closed loop appears and the $y$-axis represents the time (radius) at which it is terminated (or ``filled in"). As we increase the radius $t$ in Figure \ref{f:persistence}, many closed loops appear and quickly disappear (see, e.g., the cycles $c_1$ and $c_2$). The birth time and death time of these cycles are so close that they are plotted near the diagonal line (see points $a_1$ and $a_2$ in Figure \ref{f:diagram-plot}). The points near the diagonal line are generated by non-robust cycles and thus viewed as ``topological noise." However, cycle $c_3$ is considerably more essential and robust to the change in the value of $t$. In Figure \ref{f:diagram-plot}, the corresponding birth-death plot is located far from the bulk of the other points (see point $a_3$).

\begin{figure}[!t]
\begin{center}
\includegraphics[width=11.4cm]{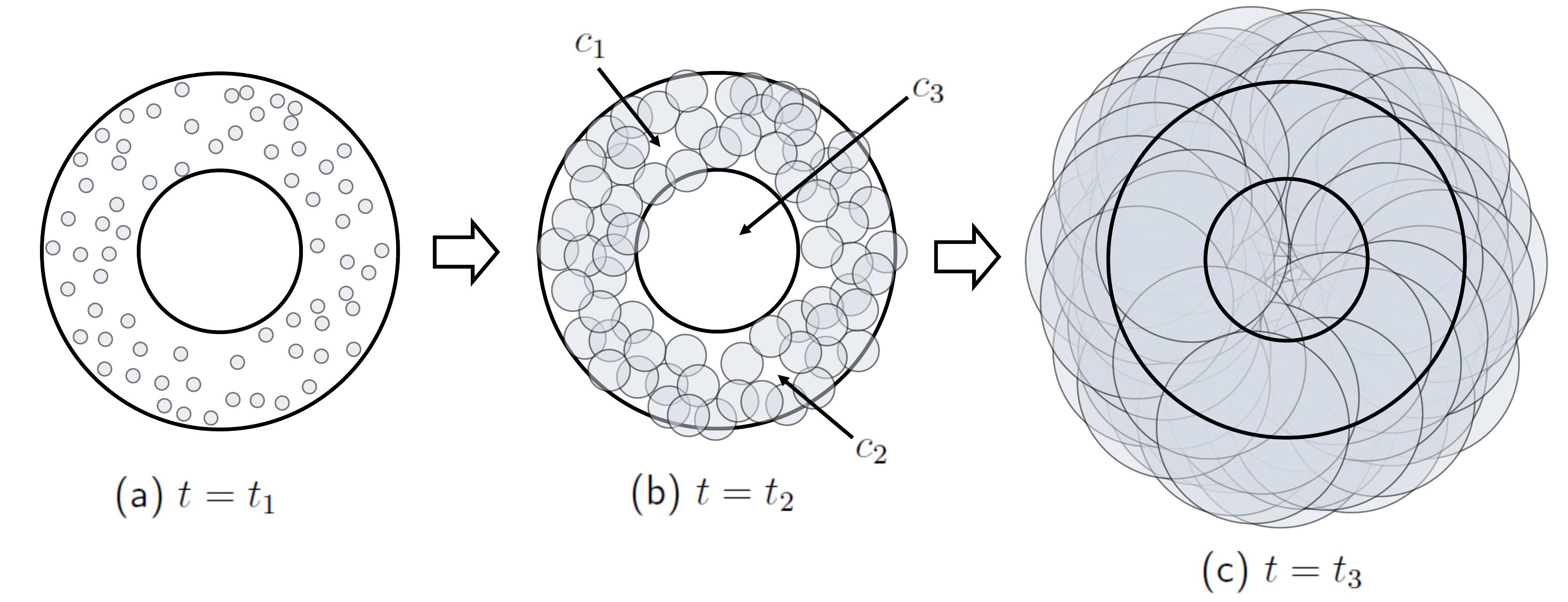}
\caption{{\footnotesize Random points are scattered in an annulus. We increase the radius $t$. This figure is taken from \cite{owada:2018}.  }}
\label{f:persistence}
\end{center}
\end{figure}
\begin{figure}[!t]
\begin{center}
\includegraphics[width=6cm]{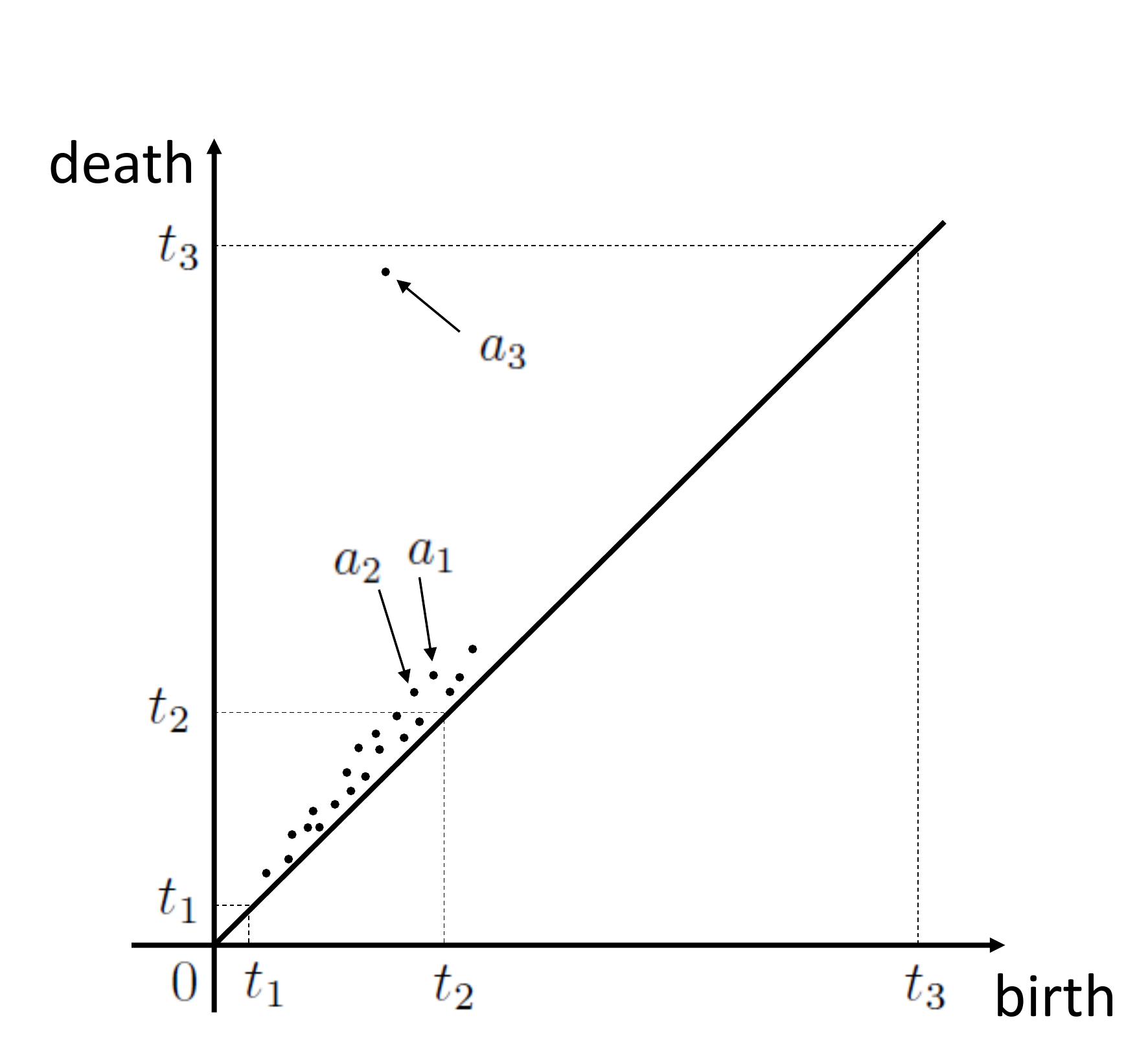}
\caption{{\footnotesize Persistence diagram for the first homology group $H_1$ represented by one-dimensional closed loops. This figure is also taken from \cite{owada:2018}. }}
\label{f:diagram-plot}
\end{center}
\end{figure}

\vspace{5pt}

In applied topology, we frequently place special emphasis on the combinatorial nature of geometric objects by means of a higher-dimensional notion of graphs, called the \emph{geometric complex}. Among many candidates of geometric complexes (see \cite{ghrist:2014}), we especially focus on the \emph{\v{C}ech complex}.
\begin{definition}
Given a set $\X=\{ x_1,\dots,x_n \}$ of points in $\bbr^d$ and a positive number $r>0$, the \v{C}ech complex $\C(\X,r)$ is defined as follows.
\begin{enumerate}
\item The $0$-simplices are the points in $\X$.
\item For each $m\ge 1$, $[ x_{i_0},\dots,x_{i_m} ]\subset \X$ forms an $m$-simplex if $\bigcap_{j=0}^{m} B(x_{i_j},r/2) \neq \emptyset$.
\end{enumerate}
\end{definition}
The main advantage of the \v{C}ech complex is its homotopy equivalence to the union of balls $U(r/2)$. This fact is known as the Nerve lemma (e.g., Theorem 10.7 in \cite{bjorner:1995}).

In the present work, we are interested in the asymptotic behavior of persistence diagrams associated with random \v{C}ech complexes. More specifically, let $\X_n=\{ X_1,\dots,X_n \}$ be an iid random sample on $\bbr^d$ with common density $f$. Then, we define the \emph{\v{C}ech filtration} by
\begin{equation}  \label{e:cech.filtration}
\mathcal C(r_n^{-1}\X_n):= \big(  \C (r_n^{-1}\X_n, t), \, t \ge 0\big) = \big(  \C (\X_n, r_nt), \, t \ge 0\big),
\end{equation}
over a ``scaled" random sample $r_n^{-1}\X_n:= \{  r_n^{-1}X_1, \dots, r_n^{-1}X_n\}$, such that $r_n\to0$ as $n\to\infty$.
Note that \eqref{e:cech.filtration} represents a nested sequence of \v{C}ech complexes, satisfying monotonicity property
$$
\C(\X_n, r_ns) \subset \C(\X_n, r_nt)  \ \ \text{for all } 0 < s \le t.
$$

In many of the studies on the stochastic topology of geometric complexes, (persistent) Betti numbers have been employed as a good quantifier of  topological complexity \cite{kahle:2011, kahle:meckes:2013, yogeshwaran:adler:2015, yogeshwaran:subag:adler:2017, bobrowski:mukherjee:2015, hiraoka:shirai:trinh:2018, krebs:polonik:2019}.
Given an integer $k\ge 0$, let $Z_k\big( \C(\X_n, r_nt) \big) = \text{ker } \partial_k$ be the $k$th cycle group of the \v{C}ech complex in \eqref{e:cech.filtration}, where $\partial_k$ is a boundary homomorphism. Additionally, denote by $B_k\big( \C(\X_n, r_nt) \big) = \text{im }\partial_{k+1}$ the $k$th boundary group of the same complex. Then, the $k$th homology group $H_k\big(  \C(\X_n,r_nt) \big) := Z_k\big( \C(\X_n, r_nt) \big)  / B_k\big( \C(\X_n, r_nt) \big)$,
is defined as the quotient group, representing the elements of (non-trivial) $k$-dimensional cycles as a boundary of a $(k+1)$-dimensional body. Hereinafter, we call it a ``$k$-cycle" for short. The \emph{$k$th Betti number}, denoted by
\begin{equation}  \label{e:def.kth.Betti}
\betakn(t):= \beta_k\big( \C(\X_n, r_nt) \big),
\end{equation}
is the rank of $H_k\big( \C(\X_n, r_nt) \big)$, representing the number of $k$-cycles in the \v{C}ech complex.
More generally, the \emph{$k$th \emph{persistent Betti number}} of $\mathcal C(r_n^{-1}\X_n)$ is defined as
\begin{equation}  \label{e:def.persistent.Betti.intro}
\betakn(s,t) := \beta_k^{s,t}\big( \mathcal C(r_n^{-1}\X_n) \big) = \text{dim} \frac{Z_k\big( \C(\X_n, r_ns) \big)}{Z_k\big( \C(\X_n, r_ns) \big) \cap B_k \big( \C(\X_n, r_nt) \big) }, \ \ 0 \le s \le t <\infty.
\end{equation}
More intuitively, $\betakn(s,t)$ represents the number of $k$-cycles that appear in $\mathcal C(r_n^{-1}\X_n)$ before time $s$ and remain alive at time $t$. In particular, if $s=t$, then $\betakn(t,t)$
reduces to the usual Betti number in \eqref{e:def.kth.Betti}. A more rigorous coverage of these algebraic topological notions can be found in \cite{munkres:1996, hatcher:2002, edelsbrunner:letscher:zomorodian:2002}.

The earliest results of examining persistence diagrams from a probabilistic viewpoint were presented in \cite{mileyko:mukherjee:harer:2011}, \cite{turner:mileyko:mukherjee:harer:2014}, the authors of which developed the definition of probability measures that support expectations, variances, and conditional probabilities. In more recent years, in \cite{chazal:divol:2018} it was shown that, for a wide class of filtrations, the expected persistence diagram as a Radon point measure has a density 
with respect to the Lebesgue measure.
From the viewpoints of stochastic topology, the studies most relevant to this paper are described in \cite{hiraoka:shirai:trinh:2018, trinh:2020, divol:polonik:2019}, in which the authors considered persistence diagrams as a point process and related them to the persistent Betti numbers.
In particular, \cite{hiraoka:shirai:trinh:2018} provided a strong law of large numbers (SLLN) for the persistence diagram generated by a stationary point process in the so-called \emph{critical (or thermodynamic) regime}. %while \cite{trinh:2020} and \cite{divol:polonik:2019} assume that the point set is  based on an iid random sample in the same regime.
In our notation, the critical regime is understood as the condition $nr_n^d\to c$ as $n\to\infty$ for some constant $c\in (0,\infty)$. In this setting, \v{C}ech complexes become highly connected, forming many large components of cycles of various dimensions \cite{yogeshwaran:subag:adler:2017, owada:thomas:2020, goel:duy:tsunoda:2019}.

In contrast to the previous works cited in the last paragraph, the main focus of this paper is the \emph{sparse regime}: $nr_n^d\to0$ as $n\to
\infty$. %and our goal is to establish the comprehensive results on the asymptotic theory of persistence diagrams.
Then, the spatial distribution of \v{C}ech complexes is more sparse and less complicated than the critical regime because of a faster decay of $r_n$ as a result of $nr_n^d\to0$.
The main point of the current paper is that, for every $k\ge 1$,
the behavior of the persistence diagram, associated with the $k$-cycles in the \v{C}ech filtration, splits into three different regimes:
\begin{equation}  \label{e:three.regimes}
(i) \ n^{k+2}r_n^{d(k+1)}\to \infty,  \ \ (ii) \ n^{k+2}r_n^{d(k+1)}\to c\in (0,\infty), \ \ (iii) \ n^{k+2}r_n^{d(k+1)}\to0.
\end{equation}
More specifically, the rate of the sequence $\seq$ determines various phase transitions of the $k$th homology group of the \v{C}ech complex $\C(\X_n, r_nt)$. For instance, in view of the expectation of $\betakn(t)$, it holds that
\begin{equation}  \label{e:exp.betaknt}
\E\big[ \betakn(t) \big] \sim C\seq, \ \ \text{as } n\to\infty,
\end{equation}
for some $C>0$ (see \cite{kahle:2011, kahle:meckes:2013}).

In particular, if $n^{k+2}r_n^{d(k+1)} \to \infty$, as in case $(i)$, the central limit theorem holds for $\betakn(t)$; that is,
$$
\frac{\betakn(t) - \E[\betakn(t)]}{\sqrt{n^{k+2}r_n^{d(k+1)}}}
$$
converges weakly to a centered Gaussian process \cite{kahle:meckes:2013, owada:thomas:2020}.
For the asymptotics of persistence diagram, Theorem \ref{t:divergence.regime} below claims that if the persistence diagram is scaled by the sequence $n^{k+2}r_n^{d(k+1)}$, it converges to a deterministic Radon measure almost surely in the vague metric. The obtained result can be viewed as a SLLN for persistence diagrams. This also implies that the number of birth-death pairs grows at the rate of $n^{k+2}r_n^{d(k+1)}$, so that the limiting persistence diagram consists of infinitely many birth-death pairs as $n\to\infty$. As for the proof techniques, our approach relies on a fundamental relation between a persistence diagram and the persistent Betti numbers $\big(\beta_{k,n}(s,t), 0 \le s \le t \le \infty\big)$ in \eqref{e:def.persistent.Betti.intro}; that is, the former can be expressed as a simple linear combination of the latter (see \eqref{e:PD-persistent.Betti} for an explicit expression). Owing to this relation, along with technical results for an underlying vague convergence (e.g., Corollary A.3 and Proposition 3.4 in \cite{hiraoka:shirai:trinh:2018}), the SLLNs for a persistence diagram can be obtained from those for persistent Betti numbers. The main difficulty here is  that the scaler for $\betakn(s,t)$ may grow very slowly (e.g., logarithmically); in such a case, a direct use of the Borel--Cantelli lemma, together with the lower-order moment calculations, does not help to establish the required SLLN. To overcome this issue, we employ the concentration inequalities in \cite{bachmann:reitzner:2018}, which themselves were developed for analyzing Poisson $U$-statistics of the geometric configuration of a point cloud. For the application of these concentration bounds, one needs to detect appropriate subsequential upper and lower bounds for the quantities that are used to approximate persistent Betti numbers. This approach is a higher-dimensional version of a  standard technique for the theory of geometric graphs; see, e.g., Chapter 3 of the monograph \cite{penrose:2003}. 

Suppose next that $r_n$ decays faster, so that $n^{k+2}r_n^{d(k+1)} \to c \in (0,\infty)$, as in case $(ii)$ of \eqref{e:three.regimes}. Then, $\E[\betakn(t)]$ is asymptotically a positive constant. As a result, $\betakn(t)$ converges weakly  
to the difference of two dependent Poisson processes  \cite{kahle:meckes:2013, owada:thomas:2020}.   Theorem \ref{t:Poisson.regime} below reveals that the  persistence diagram  weakly converges to a limiting point process without normalization. As expected, the limiting point process possesses a Poissonian structure (see \eqref{e:characterization.limit.Poisson}). This implies that the number of cycles is not large, and asymptotically, all the $k$-cycles affecting the limiting persistence diagram are necessarily formed by components of size $k+2$ (i.e., components of the smallest size). Our proof techniques for Theorem \ref{t:Poisson.regime} are closely related to those in Theorem 5.1 of \cite{owada:thomas:2020}. By virtue of an aforementioned linear  relation between a persistence diagram and the persistent Betti number $\beta_{k,n}(s,t)$, we need to demonstrate the Poisson limit theorem for $\beta_{k,n}(s,t)$.  This  will follow from weak convergence of a point process induced by the persistent Betti number (see \eqref{e:pp.conv.for.main}). For the proof, we show directly a set of sufficient conditions provided in Theorem 3.1 of \cite{decreusefond:schulte:thaele:2016}. 

Finally, if $n^{k+2}r_n^{d(k+1)}\to 0$ as in case $(iii)$ of \eqref{e:three.regimes}, the occurrence of $k$-cycles becomes even rarer, in the sense of $\E[\betakn(t)]\to 0$, $n\to\infty$. 
Accordingly, the persistence diagram converges to the null measure (i.e., the measure assigning zeros to all measurable sets in the diagram). 
In this case, the sequence of probability distributions of a persistence diagram has to be normalized. 
The problem is however that the probability distribution of a persistence diagram is defined in the space of Radon point measures. Since this space is not locally compact,  the corresponding convergence cannot be treated in vague topology. Alternatively, we employ  the $\M_0$-topology, a standard topology that has been  used for the study of regular variation of point processes and stochastic processes \cite{hult:samorodnitsky:2010, lindskog:resnick:roy:2014, segers:zhao:meinguet:2017}.  Theorem \ref{t:vanishing.regime} below gives a more precise statement. In particular, Section \ref{sec:proof.vanishing} proves   sufficient conditions that are based on Theorem A.2 in \cite{hult:samorodnitsky:2010}, for the required convergence. 

For easy reference later on,  we refer to case $(i)$ of \eqref{e:three.regimes} as the \emph{divergence regime}, since \eqref{e:exp.betaknt} indicates that $\E\big[\betakn(t)\big]$ diverges as $n\to\infty$. We also refer to case $(ii)$ as the \emph{Poisson regime}, in that the weak limit of $\betakn(t)$ is Poisson distributed. Finally, case $(iii)$ is called the \emph{vanishing regime}, for which $\betakn(t)$ vanishes in the sense of an expectation.

The potential applications of the current study are in topological data analysis (TDA). Denote by $D_{k,n}$ the $k$th persistence diagram of \v{C}ech filtration \eqref{e:cech.filtration}.
A common practice in TDA is to transform persistence diagrams into the representation
\begin{equation}  \label{e:representation.PD}
\sum_{(b_i,d_i)\in D_{k,n}} \phi(b_i,d_i),
\end{equation}
where $(b_i,d_i)$ is the list of the $k$th persistence birth-death pairs (see Section \ref{sec:setup} for a formal definition) and $\phi$ is a real-valued function.
One example of such functions is the $\alpha$th total persistence, given by $\phi(x,y)=(y-x)^\alpha$ (see \cite{divol:polonik:2019}). In the special case $\alpha =1$, it reduces to the sum of persistence barcodes \cite{ghrist:2008, carlsson:2009}.
One of the primary benefits of our results is that one may give a probabilistic foundation to the functional in \eqref{e:representation.PD}. Indeed, when the asymptotic theory for persistence diagrams has been completed, the standard machinery via the continuous mapping theorem for different modes of convergence (see \cite{resnick:1987, kallenberg:2017, hult:lindskog:2006a}) may yield a variety of limit theorems for \eqref{e:representation.PD}.
This will be technically challenging, since many of the required functionals are not continuous, especially in the region close to the diagonal line. Nonetheless, it will be possible to overcome this difficulty by means of a well established approximation scheme that was recently developed by the authors of \cite{divol:polonik:2019} (see also \cite{divol:lacombe:2021}). This line of research remains a future topic of our research.
%Several examples of approximating functionals, not necessarily relating to persistent homology, can also be found in the PI's work \cite{owada:adler:2017, owada:bobrowski:2020}.

The remainder of this paper is structured as follows. First, Section \ref{sec:setup} provides a formal definition of a persistence diagram as a point process, expressing it as a function of persistent Betti numbers. All the main results on the limiting persistence diagram are presented in Section \ref{sec:main.results}. This section is divided into three parts, corresponding to each of the regimes in  \eqref{e:three.regimes}. All the proofs are deferred to Section \ref{sec:proof}.

Before concluding the Introduction, let us add a few more comments on our setup and assumptions. First, we assume that the density $f$ of a random sample $\X_n$ is a.e.~continuous and bounded. Although it seems possible to obtain the same results under a weaker assumption that $\int_{\bbr^d} f(x)^{2k+4}dx<\infty$, we decided to rely on stronger assumptions. By doing so, we can avoid technical arguments based on the Lebesgue differentiation theorem. Second, we consider only persistence diagrams associated with \v{C}ech filtration. However, the proposed methods seem to be applicable to other geometric complexes. In particular, in the case of a Vietoris--Rips filtration, all the results obtained can be carried over by a simple replacement of the scaler $n^{k+2}r_n^{d(k+1)}$ in \eqref{e:three.regimes} with an appropriate one.

\section{Setup}  \label{sec:setup}

We start with a formal definition of the space for persistence diagrams. First, let $[0,\infty)\times [0,\infty]$ be a product space endowed with the product topology. Define an infinite triangle in the first quadrant by
$$
\Delta := \big\{ (x,y): 0\le x \le y <\infty \big\}\cup \big\{ (x,\infty): 0 \le x <\infty \big\},
$$
and equip $\Delta$ with the relative topology of $[0,\infty)\times [0,\infty]$. Let $L=\big\{ (x,x):0\le  x <\infty \big\}$ be the diagonal line in the first quadrant. Throughout the paper, we take $\DelL:=\Delta\setminus L = \big\{(x,y): 0\le x < y\le \infty\big\}$ as an underlying space for persistence diagrams.
%Note that a collection of compact sets in $\DelL$ coincides with a family of compact sets in  $\Delta$ that do not intersect with $L$.
%Here let us mention one important property on compact sets of $\DelL$: for any compact $K\subset \DelL$, there exists a finite integer $m\ge 2$, and the increasing sequences $(s_i)_{i=1}^m$ and $(t_i)_{i=1}^{m-1}$ with
%$$
%0 \le s_1 <s_2 <t_1 <s_3 <t_2 <s_4 <\cdots < s_{i+1} <t_i <s_{i+2} < \cdots <t_{m-2} <s_m <t_{m-1} < \infty,
%$$
%such that
%\begin{equation}  \label{e:finite.cover}
%K \subset \bigcup_{i=1}^{m-1} R_{s_i s_{i+1} t_i \infty},
%\end{equation}
%where
%$$
%R_{b_1b_2b_3b_4} = [b_1,b_2) \times (b_3, b_4], \ \ \ 0\le b_1 <b_2 \le b_3 <b_4\le \infty.
%$$

Let
$$
\A = \big\{ (s,t] \times (u,v], \, [0,t]\times (u,v], \, 0\le s \le t \le u \le v \le \infty,\, t < \infty \big\},
$$
and denote an element of $\A$ by $\genR$ (if $s=0$, it represents either $(0,t]\times (u,v]$ or $[0,t]\times (u,v]$).
Next, we define
$$
D_{k,n}= D_k\big( \mathcal C(r_n^{-1}\X_n) \big) = \big\{ (b_i, d_i)\in \Delta: i=1,\dots,n_k \big\}
$$
to be the $k$th persistence diagram associated with the \v{C}ech filtration \eqref{e:cech.filtration}. Here, $\big( (b_i,d_i) \big)_{i=1}^{n_k}$ is the list of the $k$th persistence birth-death pairs, representing the time at which each $k$-cycle first appears in $\mathcal C (r_n^{-1}\X_n)$ and the time at which it is terminated (or filled in), respectively. We then define $D_{k,n}$ as a point process,
\begin{equation}  \label{e:def.PD}
\xi_{k,n} := \sum_{(b_i,d_i)\in D_{k,n}} \delta_{(b_i,d_i)},
\end{equation}
where $\delta_{(x,y)}$ is the Dirac measure at $(x,y)\in \bbr^2$. A key relation between \eqref{e:def.PD} and the persistent Betti number \eqref{e:def.persistent.Betti.intro} is that, for $0\le s \le t \le u \le v\le \infty$,
\begin{equation}  \label{e:PD-persistent.Betti}
\xi_{k,n}(\genR) = \betakn(t,u) - \betakn(t,v) - \betakn(s,u) + \betakn(s,v).
\end{equation}
Finally, we introduce a certain indicator function that is used for characterizing the limiting objects for each of the regimes in \eqref{e:three.regimes}. For $r>0$ and $(x_1,\dots,x_{k+2})\in (\bbr^d)^{k+2}$, define
\begin{align}
h_r(x_1,\dots,&x_{k+2}) := \one \Big\{  \beta_k \big(  \C( \big\{  x_1,\dots,x_{k+2} \big\},   r) \big) = 1 \Big\} \label{e:def.hr}\\
&= \one \, \bigg\{ \Big\{  \bigcap_{j=1, \, j \neq j_0}^{k+2}B(x_j, r/2) \neq \emptyset \text{ for all } j_0 \in \{ 1,\dots,k+2 \} \Big\} \cap \Big\{  \bigcap_{j=1}^{k+2}B(x_j, r/2) = \emptyset \Big\} \bigg\}, \notag
\end{align}
where $\one \{ \cdot \}$ is an indicator function. This indicator function requires that a set $\{x_1,\dots,x_{k+2}\}$ of points in $\bbr^d$ forms a single $k$-cycle with connectivity radius $r$.

\section{Main results}  \label{sec:main.results}

In this section, we present the limit theorems for persistence diagrams. For ease of discussion, we first treat the Poisson regime and then move on to the other two cases. Recall that $\X_n=\{ X_1,\dots,X_n\}$ represents iid random variables on $\bbr^d$ with common density $f$. We assume that $f$ is a.e.~continuous and bounded, that is,  $\| f\|_\infty := \text{esssup}_{x\in\bbr^d} f(x) < \infty$. Moreover, we take an integer $k\ge1$, which remains fixed in the following. Since our main focus is the sparse regime, we assume throughout the paper that $nr_n^d\to0$ as $n\to\infty$.

\subsection{Poisson regime}

Assume that $n^{k+2}r_n^{d(k+1)} \to c$ as $n\to\infty$ for some constant $c\in (0,\infty)$. Write $\lambda$ for the $d(k+1)$-dimensional Lebesgue measure and $C_k:= ((k+2)!)^{-1}\int_{\bbr^d}f(x)^{k+2}dx$.
Let $M_+(\DelL)$ be the space of Random measures on $\DelL$, which is equipped with \emph{vague topology} (see \cite{kallenberg:2017, resnick:1987}).
Additionally, denote by $M_p(\DelL)$ the subset in $M_+(\DelL)$ of all Radon point measures. Note that $M_p(\DelL)$ is a closed subset of $M_+(\DelL)$ in vague topology (see Proposition 3.14 in \cite{resnick:1987}).
Next, for $\bx=(x_1,\dots,x_{k+2}) \in (\bbr^d)^{k+2}$ we define
$$
\genH(\bx) := h_{t}(\bx)h_{u}(\bx) - h_{t}(\bx)h_{v}(\bx) - h_{s}(\bx)h_{u}(\bx) + h_{s}(\bx) h_{v}(\bx), \ \ 0 \le s\le t \le u \le v\le \infty,
$$
where the indicators on the right hand side are defined in \eqref{e:def.hr}.

Before stating the main result, we need to define the limiting point process of \eqref{e:def.PD}. First, we note that $\A$ forms a semi-ring of subsets of $\DelL$. Indeed, one can check that: $(i) \ \emptyset\in \A$; $(ii)$ if $A,B\in \A$, then $A\cap B\in \A$; and $(iii)$ if $A,B\in \A$, then $A\setminus B=\bigcup_{i=1}^m C_i$ for some finite disjoint sets $C_1,\dots,C_m\in \A$. Moreover, $\A$ has the  covering property; i.e., for any open $G\subset \DelL$, there exists $(A_i)_{i=1}^\infty \subset \A$ so that $G=\bigcup_{i=1}^\infty A_i$. In particular, $\A$ generates a Borel $\sigma$-field on $\DelL$.
It then follows from Proposition 9.2.~III.~in \cite{daley:verejones:2008} that one can define a point process $\zeta_k$, whose probability distribution on $M_p(\DelL)$ is uniquely determined by the finite-dimensional distributions
\begin{align}
&\P \big( \zeta_k (\genRi) =m_i, \ i=1,\dots,d\big) \label{e:characterization.limit.Poisson}\\
&= \P \Big( \int_{(\bbr^d)^{k+1}} \genHi(0,y_1,\dots,y_{k+1}) M_k(d\by) = m_i, \ i=1,\dots,d \Big) \notag
\end{align}
for $d\ge 1$, and $m_i\in \bbn$, $0\le s_i \le t_i \le u_i \le v_i \le \infty$, $i=1,\dots,d$. Here, $M_k$ denotes the Poisson random measure on $(\bbr^d)^{k+1}$ with mean measure $C_k \lambda$. Namely, the distribution of $M_k$ is defined as
$$
M_k(A) \sim \text{Poi}\big( C_k\lambda(A) \big)
$$
for all measurable $A \subset (\bbr^d)^{k+1}$ (``Poi" stands for a Poisson distribution). Furthermore, if $A_1,\dots,A_m$ are disjoint subsets in $(\bbr^d)^{k+1}$, then $M_k(A_1), \dots, M_k(A_m)$ are independent.

As defined, the marginal distribution of $\zeta_k$ depends on a linear combination of the indicators in \eqref{e:def.hr}. This implies that the $k$-cycles affecting $\zeta_k$ must always be formed by connected components on $k+2$ points (i.e., components of the smallest size).
It is also easy to check that $\zeta_k(\genR)$ has a Poisson law with mean $C_k\int_{(\bbr^d)^{k+1}} \genH(0,y_1,\dots,y_{k+1})d\by$. Further characterization of \eqref{e:characterization.limit.Poisson} (as a random field) is provided in \cite{owada:2018}.

\begin{theorem}  \label{t:Poisson.regime}
Suppose $n^{k+2}r_n^{d(k+1)} \to c\in (0,\infty)$ as $n\to\infty$. Then, as $n\to\infty$,
\begin{equation}  \label{e:weak.conv.xikn}
\xi_{k,n} \Rightarrow c\zeta_k \ \ \text{in } M_p(\DelL),
\end{equation}
where $\Rightarrow$ denotes weak convergence.
\end{theorem}

\subsection{Divergence regime}  \label{sec:divergence.regime}

Next, we turn to the case for which $n^{k+2}r_n^{d(k+1)} \to \infty$ as $n\to\infty$. In this case, there appear infinitely many $k$-cycles as $n\to\infty$ (see \eqref{e:exp.betaknt}), and accordingly, the limiting persistence diagram of $\xi_{k,n}$ consists of infinitely many persistence birth-death pairs, so that $\xi_{k,n}(\genR)\to \infty$ as $n\to\infty$ for all $0\le s \le t \le u \le v\le \infty$. Hence, in order to obtain a non-degenerate limit of $\xi_{k,n}$, the process itself must be normalized by a growing sequence.
Let $C_K^+(\DelL)$ be a collection of  non-negative  and continuous functions on $\DelL$ with compact support.  Recall that a sequence of Radon measures $(\nu_n)\subset M_+(\DelL)$ is said to converge vaguely to $\nu \in M_+(\DelL)$, denoted by $\nu_n\stackrel{v}{\to}\nu$ in $M_+(\DelL)$, if $\int_{\DelL}f(x)\nu_n(dx) \to \int_{\DelL}f(x)\nu(dx)$ for all $f\in C_K^+(\DelL)$.

For the theorem below, we need to 
%check apply a mild condition to the decay rate?
apply a mild condition to the decay rate of $r_n$. More precisely, we assume that it is a regularly varying sequence (at infinity) with exponent $\rho<0$,
$$
\lim_{n\to\infty}\frac{r_{\lfloor an \rfloor}}{r_n} = a^\rho \ \ \text{ for all } a>0.
$$

Finally, for two sequences $(a_n)$ and $(b_n)$, write $a_n = \Omega(b_n)$ if there exists a constant $C>0$ such that $a_n/b_n \ge C$ for all $n \ge 1$. %Write also $a_n=\Omega (b_n)$ if there exists $C'>0$ so that $a_n/b_n \ge C'$ for all $n\ge 1$.
\begin{theorem}  \label{t:divergence.regime}
Suppose  $n^{k+2}r_n^{d(k+1)} \to \infty$ as $n\to\infty$. Assume that $(r_n)$ is a regularly varying sequence with exponent $\rho<0$, such that
\begin{equation}  \label{e:faster.than.log}
n^{k+2}r_n^{d(k+1)} = \Omega \big( (\log n)^\eta \big)
\end{equation}
for some $\eta > 0$.
Then, there exists a unique Radon measure $\mu_k\in M_+(\DelL)$ such that
\begin{equation}  \label{e:vague.as.conv}
\frac{\xi_{k,n}}{n^{k+2}r_n^{d(k+1)}} \stackrel{v}{\to} \mu_k, \ \ \text{almost surely in } M_+(\DelL),
\end{equation}
and
\begin{equation}  \label{e:exp.vague.conv}
\frac{\E[\xi_{k,n}]}{n^{k+2}r_n^{d(k+1)}} \stackrel{v}{\to} \mu_k, \ \ \text{in } M_+(\DelL),
\end{equation}
where $\E[\xi_{k,n}]$ denotes the mean measure of $\xi_{k,n}$, which itself is a Radon measure on $\DelL$. The measure $\mu_k$ satisfies
\begin{equation}  \label{e:limit.measure.divergence}
\mu_k(\genR) = C_k \int_{(\bbr^d)^{k+1}} \genH (0,y_1,\dots,y_{k+1})d\by,
\end{equation}
except for at most countably many $\genR\in \A$.
\end{theorem}

In spite of the significant difference in the proof techniques, the results above more or less parallel Theorem 1.5 in \cite{hiraoka:shirai:trinh:2018} (see also Theorem 1.1 in \cite{trinh:2020}), in which the authors obtained the SLLN for persistence diagrams in the critical regime. In the critical regime, \v{C}ech complexes are highly connected, so that they form many large components of cycles of any dimension. As a consequence, the limiting measures in \cite{hiraoka:shirai:trinh:2018, trinh:2020} are so complicated that they do not have an explicit representation, as in \eqref{e:limit.measure.divergence}.

\subsection{Vanishing regime}

Once again, we return to \eqref{e:three.regimes} and consider case $(iii)$ for which $n^{k+2}r_n^{d(k+1)} \to 0$ as $n\to\infty$.
In this case, it holds that $\E\big[ \betakn(t) \big] \sim C n^{k+2}r_n^{d(k+1)} \to 0$ as $n\to\infty$. This indicates that all the $k$-cycles vanish in the limit, and accordingly, $\xi_{k,n}$ converges to the null measure $\emptyset \in M_p(\DelL)$ (i.e., the measure assigning zeros to all measurable sets in $\DelL$). In particular, we have, for all $m=1,2,\dots$,
$$
\lim_{n\to\infty} \P \big( \xi_{k,n}(\genR) = m \big) = 0, \ \ 0\le s \le t \le u \le v \le \infty.
$$
In the stochastic topology literature, not necessarily related to random geometric complexes, a fundamental interest lies in how rapidly each homology group appears and disappears \cite{bobrowski:weinberger:2017, fowler:2019, kahle:2014, kahle:pittel:2016, skraba:thoppe:yogeshwaran:2020}. In the same spirit, we are naturally interested in the decay rate of the sequence $(\P\circ \xi_{k,n}^{-1})$ of probability measures on $M_p(\DelL)$. However, the space $M_p(\DelL)$ is not locally compact, and thus, the resulting convergence can no longer be treated under vague topology, as in the last section. To overcome this difficulty, we adopt a notion of \emph{$\M_0$-topology}. This notion was first developed by the authors of \cite{hult:lindskog:2006a}. Since then, it has been intensively used, especially in extreme value theory, for the study of the regular variation of point processes and stochastic processes \cite{hult:samorodnitsky:2010, lindskog:resnick:roy:2014, segers:zhao:meinguet:2017}.
The main benefit of employing $\M_0$-topology is that it requires only that the underlying space be complete and separable. Since $M_p(\DelL)$ is complete and separable (see Proposition 3.17 in \cite{resnick:1987}), we can utilize $\M_0$-topology as an appropriate topology for the convergence below.

Let $B_{\emptyset, r}$ denote an open ball of radius $r>0$ centered at the null measure $\emptyset\in M_p(\DelL)$ in terms of the metric induced by vague topology. Denote by $\M_0 = \M_0\big( M_p(\DelL) \big)$ the space of Borel measures on $M_p(\DelL)$, the restriction of which to
 $M_p(\DelL)\setminus \ B_{\emptyset, r}$ is finite for all $r>0$. Moreover, define $\mathcal C_0  =\mathcal C_0\big( M_p(\DelL) \big)$ to
 be the space of continuous and bounded real-valued functions on $M_p(\DelL)$ that vanish in the neighborhood of $\emptyset$. Finally, given $m_n, m \in \M_0$, the convergence $m_n\to m$ in $\M_0$ is defined as $\int_{M_p(\DelL)}f(x)m_n(dx) \to \int_{M_p(\DelL)}f(x)m(dx)$ for all $f\in \mathcal C_0$.

\begin{theorem}  \label{t:vanishing.regime}
Suppose $n^{k+2}r_n^{d(k+1)} \to 0$ as $n\to\infty$. Then, as $n\to\infty$,
\begin{align*}
(n^{k+2}r_n^{d(k+1)})^{-1} \P (\xi_{k,n}\in \cdot) \to C_k \lambda \Big\{ \by\in (\bbr^d)^{k+1}: \delta_{( b(0,\by), d(0,\by))} \in \cdot \Big\} \ \ \text{in } \M_0,
\end{align*}
where $\by=(y_1,\dots,y_{k+1}) \in (\bbr^d)^{k+1}$ and for $\bx=(x_1,\dots,x_{k+2})\in (\bbr^d)^{k+2}$,
\begin{align*}
b(\bx) &= \inf\big\{ t\ge 0: h_t (\bx)=1 \big\},   \\
d(\bx) &= \inf\big\{ t\ge b(\bx): h_t (\bx)=0\big\}
\end{align*}
(by convention, we take $\inf \emptyset \equiv \infty$).
\end{theorem}

\begin{remark}
We here provide a more detailed structure of the limiting measure
$$
\eta(\cdot) := C_k \lambda \Big\{ \by\in (\bbr^d)^{k+1}: \delta_{( b(0,\by), d(0,\by))} \in \cdot \Big\}.
$$
For $m\ge 1$ and $0 \le s_i \le t_i \le u_i \le v_i \le \infty$, $i=1,\dots,m$, we define
a map $T:M_p(\DelL) \to [0,\infty)^m$ by $T(\xi)=\big( \xi(\genRi)\big)_{i=1}^m$. It is then straightforward to check that
$$
\eta\circ T^{-1} (\cdot)= C_k \lambda \Big\{ \by \in (\bbr^d)^{k+1}: \Big( \one \big\{  (b(0,\by), d(0,\by)) \in \genRi \big\} \Big)_{i=1}^m \in \cdot \Big\}.
$$
Observe also that, for every $i\in \{ 1,\dots,m \}$, $\big( b(0,\by), d(0,\by) \big) \in \genRi$ holds if and only if
\begin{equation}  \label{e:4cond}
h_{s_i}(0,\by) = 0, \ \ h_{t_i}(0,\by) = h_{u_i}(0,\by) = 1, \ \ h_{v_i}(0,\by) = 0.
\end{equation}
In particular, \eqref{e:4cond} requires that a point set $\{ 0,\by \} = \{ 0,y_1,\dots,y_{k+1} \}\in (\bbr^d)^{k+2}$ forms a single $k$-cycle between times $s_i$ and $t_i$, such that this cycle disappears between times $u_i$ and $v_i$.
Combining these observations, we finally obtain that
$$
\eta\circ T^{-1}(\cdot) = C_k \lambda \Big\{ \by \in (\bbr^d)^{k+1}: \Big( \big(  1-h_{s_i}(0,\by)\big) h_{t_i}(0,\by)h_{u_i}(0,\by) \big(  1-h_{v_i}(0,\by)\big)\Big)_{i=1}^m \in \cdot \Big\}.
$$
%which especially yields that
%$$
%m\circ T^{-1} \big( \{1,\dots,1\} \big) =C_k\int_{(\bbr^d)^{k+1}} \prod_{i=1}^m \big(  1-h_{s_i}(0,\by)\big) h_{t_i}(0,\by)h_{u_i}(0,\by) \big(  1-h_{v_i}(0,\by)\big)d\by.
%$$
\end{remark}

\section{Proof}   \label{sec:proof}
This section presents the proofs of all the main theorems in Section \ref{sec:main.results}. First, for a set $\Y$ of $k+2$ points in $\bbr^d$ and a finite set $\Z \supset\Y$ in $\bbr^d$, and $0\le r_1 \le r_2\le \infty$,  we define an indicator function,
$$
g_{r_1,r_2}(\Y,\Z) := h_{r_1}(\Y)h_{r_2}(\Y)\, \one \big\{  \C(\Y, r_2) \text{ is a connected component of } \C(\Z, r_2) \big\}.
$$

Below, we first present a preparatory lemma, which claims that the sum of these indicators can approximate the persistent Betti numbers. As mentioned in Section \ref{sec:main.results}, we assume that the density $f$ of a random sample $\X_n$ is a.e.~continuous and bounded. Moreover, we assume that $nr_n^d \to 0$ as $n\to\infty$.

\begin{lemma} \label{l:Betti.bounds}
For all $0\le s \le t\le \infty$,
\begin{equation}  \label{e:persistent.Betti.bound}
\sum_{\substack{\Y\subset \X_n, \\ |\Y|=k+2}} g_{r_ns, r_nt}(\Y,\X_n) \le \betakn(s,t) \le \sum_{\substack{\Y\subset \X_n, \\ |\Y|=k+2}} g_{r_ns, r_nt}(\Y,\X_n)  +\binom{k+3}{k+1}L_{r_nt},
\end{equation}
where
$$
L_r = \sum_{\substack{\Y\subset \X_n, \\ |\Y|=k+3}} \one \big\{ \C(\Y,r) \text{ is connected} \big\}, \ \ \ r>0.
$$
Moreover, we have for all $0<t<\infty$
\begin{equation}  \label{e:conv.Lr}
(n^{k+2}r_n^{d(k+1)})^{-1} \E[L_{r_nt}] \to 0, \ \ \ n\to\infty.
\end{equation}
\end{lemma}
\begin{proof}
We observe that the leftmost term in \eqref{e:persistent.Betti.bound} represents the number of $k$-cycles built over $k+2$ points that are born before time $r_ns$ and still alive at time $r_nt$, such that each of these $k$-cycles is isolated from the remaining points at time $r_nt$. In contrast, $\betakn(s,t)$ counts all $k$-cycles on $i$ points for all possible $i \ge k+2$ that are born before time $r_ns$ and still alive at time $r_nt$. Hence, $\betakn(s,t)$ counts more $k$-cycles than does the leftmost term in \eqref{e:persistent.Betti.bound}; thus, the inequality on the left hand side of \eqref{e:persistent.Betti.bound} has been obtained.

For the remaining inequality in \eqref{e:persistent.Betti.bound}, one needs an explicit representation of $\betakn(t)$.  For $(x_1,\dots,x_i)\in (\bbr^d)^i$ with $i \ge k+2$, $j\ge 1$, and $r>0$,  we define
$$
\hij_r (x_1,\dots,x_i) := \one \Big\{ \beta_k \big( \C\big( \{x_1,\dots,x_i \}, r\big) \big) = j, \, \, \C\big( \{x_1,\dots,x_i\}, r \big) \text{ is connected} \Big\}.
$$
Clearly, $h_r^{(k+2,1)}(x_1,\dots,x_{k+2}) = h_r(x_1,\dots,x_{k+2})$. Moreover, for a set $\Y$ of $i$ points in $\bbr^d$ and a finite set $\Z \supset\Y$ in $\bbr^d$, and $r>0$, we define
$$
\gij_r(\Y, \Z) := \hij_r(\Y)\, \one \big\{ \C(\Y, r) \text{ is a connected component of } \C(\Z, r) \big\}.
$$
Note that $h_r^{(k+2,j)}(\Y)$ and $g_r^{(k+2,j)}(\Y,\X_n)$ are identically zero for all $j\ge 2$, since it is impossible to form multiple $k$-cycles from $k+2$ points.
With these indicators available, $\betakn(t)$ can be represented as
\begin{equation}  \label{e:representation.kth.Betti}
\betakn(t) = \beta_k\big( \C(\X_n, r_nt) \big) = \sum_{i=k+2}^n \sum_{j>0} j \sum_{\substack{\Y\subset \X_n, \\  |\Y|=i}} \gij_{r_nt}(\Y, \X_n).
\end{equation}

Using this representation, we claim that
\begin{equation}  \label{e:first.claim.ineq}
\betakn(s,t) \le \sum_{\substack{\Y\subset \X_n, \\ |\Y|=k+2}}g_{r_ns, r_nt}(\Y, \X_n) + \sum_{i=k+3}^n \sum_{j>0} j \sum_{\substack{\Y\subset \X_n, \\  |\Y|=i}} \gij_{r_nt}(\Y, \X_n).
\end{equation}
In fact, by \eqref{e:representation.kth.Betti} the right hand side of \eqref{e:first.claim.ineq} is equal to
\begin{align*}
&\sum_{\substack{\Y\subset \X_n, \\ |\Y|=k+2}}g_{r_ns, r_nt}(\Y, \X_n) +\betakn(t) - \sum_{\substack{\Y\subset \X_n, \\ |\Y|=k+2}}g_{ r_nt}^{(k+2,1)}(\Y, \X_n) \\
&= \betakn(t) - \sum_{\substack{\Y\subset \X_n, \\ |\Y|=k+2}} \big( 1-h_{r_ns}(\Y) \big) g_{r_nt}^{(k+2,1)} (\Y, \X_n) =: \betakn(t) -A_n(s,t).
\end{align*}
Here, $A_n(s,t)$ represents the number of $k$-cycles on $k+2$ points that are born between times $r_ns$ and $r_nt$ and still alive at time $r_nt$, such that each of these $k$-cycles is isolated from the remaining points at time $r_nt$. Hence, by the definition of $\betakn(s,t)$, it holds that $\betakn(t) -A_n(s,t) \ge \betakn(s,t)$.

By \eqref{e:first.claim.ineq}, it now remains to show that
\begin{equation}  \label{e:second.claim.ineq}
\sum_{i=k+3}^n \sum_{j>0} j \sum_{\substack{\Y\subset \X_n, \\  |\Y|=i}} \gij_{r_nt}(\Y, \X_n) \le \binom{k+3}{k+1} L_{r_nt}.
\end{equation}
To prove this, rewrite the left hand side of \eqref{e:second.claim.ineq} as
\begin{align}
&\sum_{i=k+3}^n \sum_{j>0} j \sum_{\substack{\Y\subset \X_n, \\  |\Y|=i}} \gij_{r_nt}(\Y, \X_n) \label{e:rewrite.above} \\
&= \sum_{i=k+3}^n \sum_{\substack{\Y\subset \X_n, \\  |\Y|=i}} \beta_k\big( \C(\Y, r_nt) \big) \one \big\{ \C(\Y, r_nt) \text{ is a connected component of } \C(\X_n, r_nt) \big\}. \notag
\end{align}
Here, we recall that $\betakn\big( \C(\Y, r_nt) \big)$ is bounded by the number of $k$-simplices on $\Y$ (with connectivity radius $r_nt$). Denote such $k$-simplex counts by $f_{k,r_nt}(\Y)$.

Suppose now that  $\C(\Y, r_nt)$ is a connected component of $\C(\X_n, r_nt)$ for some $i\ge k+3$ and $\Y\subset \X_n$ with $|\Y|=i$. Then, there exists a point set $\Z\subset \Y$ with $|\Z|=k+3$, such that $\C(\Z, r_nt)$ is a \emph{connected} subcomplex of $\C(\Y, r_nt)$.
%check Every time such a connected subcomplex occurs,
Every time such a connected subcomplex occurs, it increases $f_{k,r_nt}(\Y)$ by at most $\binom{k+3}{k+1}$.
Note also that all $k$-simplices in $\Y$ are necessarily contained in one such connected subcomplex on $k+3$ points. In conclusion, for each $i\ge k+3$ and $\Y\subset \X_n$ with $|\Y|=i$ such that $\C(\Y,r_nt)$ is a connected component of $\C(\X_n, r_nt)$, we have that
\begin{align*}
\beta_k\big( \C(\Y, r_nt) \big) \le f_{k,r_nt}(\Y)\le  \binom{k+3}{k+1} \sum_{\substack{\Z\subset\Y, \\ |\Z|=k+3}} \one \big\{ \C(\Z, r_nt) \text{ is connected} \big\}.
\end{align*}
Applying this bound to \eqref{e:rewrite.above},
\begin{align*}
&\sum_{i=k+3}^n \sum_{j>0} j \sum_{\substack{\Y\subset \X_n, \\  |\Y|=i}} \gij_{r_nt}(\Y, \X_n) \\
&\le \binom{k+3}{k+1}\sum_{i=k+3}^n \sum_{\Y\subset \X_n, \,  |\Y|=i} \one \big\{ \C(\Y, r_nt) \text{ is a connected component of } \C(\X_n, r_nt) \big\} \\
&\qquad \qquad \qquad \qquad \qquad \qquad  \times \sum_{\Z\subset\Y, \,  |\Z|=k+3} \one \big\{ \C(\Z, r_nt) \text{ is connected} \big\} \\
&= \binom{k+3}{k+1} L_{r_nt},
\end{align*}
where the last equality is due to the fact that no two different connected components of $\C(\X_n, r_nt)$ can contain the same connected subcomplex on $k+3$ vertices.

We now proceed to showing \eqref{e:conv.Lr}. Write
$$
\E[L_{r_nt}] = \binom{n}{k+3} \int_{(\bbr^d)^{k+3}} \one \big\{ \C(\{ x_1,\dots,x_{k+3}\}, r_nt) \text{ is connected}\big\} \prod_{i=1}^{k+3} f(x_i) d\bx.
$$
By the change of variables $x_i = x+r_ny_{i-1}$, $i=1,\dots,k+3$ (with $y_0\equiv 0$), together with the translation invariance of \eqref{e:def.hr},
\begin{align*}
\E[L_{r_nt}] &= \binom{n}{k+3} r_n^{d(k+2)} \int_{\bbr^d}  \int_{(\bbr^d)^{k+2}} \one \big\{ \C(\{ 0,y_1,\dots,y_{k+2}\}, t) \text{ is connected}\big\} \\
&\qquad \qquad \qquad \qquad \qquad \qquad \qquad \qquad \qquad \times f(x) \prod_{i=1}^{k+2}f(x+r_ny_i)d\by dx \\
&\le \frac{\|f\|_\infty^{k+2}}{(k+3)!}\, n^{k+3}r_n^{d(k+2)} \int_{(\bbr^d)^{k+2}} \one \big\{ \C(\{ 0,y_1,\dots,y_{k+2}\}, t) \text{ is connected}\big\} d\by.
\end{align*}
Since $\|f \|_\infty <\infty$ and $nr_n^d \to 0$ as $n\to\infty$, we obtain \eqref{e:conv.Lr}.
\end{proof}

\subsection{Proof of Theorem \ref{t:Poisson.regime}}

Without loss of generality, we prove only the case $c=1$. Furthermore, for simplicity in our proof, we may and do assume that $n^{k+2}r_n^{d(k+1)}=1$. From Proposition 11.1.~VIII.~(iv) of \cite{daley:verejones:2008}, \eqref{e:weak.conv.xikn} follows if one can verify the following two conditions. 
First, one needs to show that
 $\genR$ is a continuity set for $\zeta_k$ for all $0 \le s \le t \le u \le v \le \infty$; that is,
\begin{equation}  \label{e:first.requirement}
\P \big( \zeta_k(\partial \genR) = 0 \big)=1,
\end{equation}
 ($\partial A$ is a boundary of $A$). The second requirement for \eqref{e:weak.conv.xikn} is that
\begin{equation}  \label{e:second.requirement}
\big( \xi_{k,n}(\genRi), \, i=1,\dots,m \big) \Rightarrow \big( \zeta_k(\genRi), \, i=1,\dots,m \big)  \ \ \text{in } [0,\infty)^m
\end{equation}
for all $m\ge 1$ and $0 \le s_i \le t_i \le u_i \le v_i \le \infty$, $i=1,\dots,m$.
Let us first prove \eqref{e:first.requirement}. Note that
\begin{equation}  \label{e:boundary.upper.bound}
\partial \genR \subset \big( \{s\}\times [u,v] \big) \cup \big( \{t\}\times [u,v] \big) \cup \big( [s,t]  \times \{ u\} \big) \cup \big( [s,t]\times \{ v\} \big).
\end{equation}
Now, we show that
\begin{equation}  \label{e:zero.prob.boundary}
\P\Big( \zeta_k\big( \{s\}\times [u,v] \big)=0 \Big)=1, \ \ 0 \le s \le u \le v \le \infty.
\end{equation}
We check only the case $s>0$. Writing
$$
\{s\}\times [u,v] = \bigcap_{n=1}^\infty R_{s-n^{-1}, s+n^{-1}, u-n^{-1}, v},
$$
and appealing to Markov's inequality,
\begin{align*}
\P\Big( \zeta_k\big( \{s\}\times [u,v] \big)\ge1 \Big) &\le \E \Big[  \zeta_k \Big( \bigcap_{n=1}^\infty R_{s-n^{-1}, s+n^{-1}, u-n^{-1}, v} \Big) \Big] =\E \Big[ \lim_{n\to\infty} \zeta_k \big( R_{s-n^{-1}, s+n^{-1}, u-n^{-1}, v} \big) \Big].
\end{align*}
Applying Fatou's lemma together with \eqref{e:characterization.limit.Poisson}, we have that
\begin{align*}
\P\Big( \zeta_k\big( \{s\}\times [u,v] \big)\ge1 \Big) &\le \liminf_{n\to\infty} \E \Big[ \zeta_k \big( R_{s-n^{-1}, s+n^{-1}, u-n^{-1}, v} \big) \Big] \\
&=\liminf_{n\to\infty} \E\Big[ \int_{(\bbr^d)^{k+1}} H_{s-n^{-1}, s+n^{-1}, u-n^{-1}, v} (0,\by) M_k (d\by) \Big] \\
&=C_k \liminf_{n\to\infty}  \int_{(\bbr^d)^{k+1}} H_{s-n^{-1}, s+n^{-1}, u-n^{-1}, v} (0,\by)d\by.
\end{align*}
By the dominated convergence theorem, the last term is equal to $0$, and hence, we have obtained \eqref{e:zero.prob.boundary}. Repeating the same arguments for each of the terms in the right hand side of \eqref{e:boundary.upper.bound}, we obtain \eqref{e:first.requirement}, as desired.

Now, we proceed to show the weak convergence in \eqref{e:second.requirement}. Our proof techniques are  related to those in Theorem 5.1 of \cite{owada:thomas:2020}. Because of \eqref{e:PD-persistent.Betti} and \eqref{e:characterization.limit.Poisson}, it suffices to verify that, as $n\to\infty$,
\begin{equation}  \label{e:persistent.Betti.weak.conv}
\big( \betakn(s_i, t_i), \, i=1,\dots,m \big) \Rightarrow \Big( \int_{(\bbr^d)^{k+1}}h_{s_i}(0,\by)h_{t_i}(0,\by) M_k (d\by), \, i=1,\dots,m \Big)
\end{equation}
for all $m\ge 1$ and $0\le s_i \le t_i \le \infty$, $i=1,\dots, m$, where $\by=(y_1,\dots,y_{k+1})\in (\bbr^d)^{k+1}$. 
In the following, we first demonstrate that
\begin{equation}  \label{e:def.Gknst}
\big( G_{k,n}(s_i, t_i), \, i=1,\dots, m \big) := \Big( \sum_{\Y\subset \X_n, \, |\Y|=k+2} h_{r_ns_i}(\Y) h_{r_nt_i}(\Y), \, i=1,\dots,m \Big)
\end{equation}
converges weakly to the limit process in \eqref{e:persistent.Betti.weak.conv}, and then, it is shown that the difference between \eqref{e:def.Gknst} and the persistent Betti numbers in \eqref{e:persistent.Betti.weak.conv} vanishes in probability as $n\to\infty$.

Now, we claim that for every $m\ge 1$ and $0 \le s_i \le t_i \le \infty$, $a_i\ge 0$, $i=1,\dots, m$,
\begin{equation}  \label{e:pp.conv.for.main}
\eta_{k,n}:= \sum_{\Y\subset \X_n, \, |\Y|=k+2}\delta_{\sum_{i=1}^m a_i h_{r_ns_i}(\Y)h_{r_nt_i}(\Y)} \Rightarrow \eta_k \ \ \text{in } M_p\big((0,\infty]  \big).
\end{equation}
Here, $\eta_k$ is the Poisson random measure on $(0,\infty]$ with finite mean measure $C_k \tau_k$, where
$$
\tau_k (\cdot) = \lambda \Big\{  \by \in (\bbr^d)^{k+1}:  \sum_{i=1}^m a_ih_{s_i}(0,\by)h_{t_i}(0,\by)\in \cdot  \Big\}.
$$
Notice that $\eta_k$ can be represented as $\eta_k=\sum_{i=1}^M\delta_{Y_i}$, where $(Y_i)$ is a sequence of iid random variables with distribution $\tau_k(\cdot)/\tau_k\big( (0,\infty] \big)$ and $M$ is a Poisson random variable with mean $C_k\tau_k\big( (0,\infty] \big)$, so that $(Y_i)$ and $M$ are independent.

Now, we prove \eqref{e:pp.conv.for.main}. After establishing it, we will show the weak convergence of \eqref{e:def.Gknst} via the continuous mapping theorem, together with technical approximation arguments.
By virtue of Theorem 3.1 in  \cite{decreusefond:schulte:thaele:2016}, the following two conditions suffice for \eqref{e:pp.conv.for.main}. The first requirement is the convergence in terms of the total variation distance,
\begin{equation}  \label{e:total.variation}
\sup_A \big| \, \E [\eta_{k,n}(A)] - \E[\eta_k(A)] \, \big| \to 0, \ \ \ n\to\infty,
\end{equation}
where the supremum is taken over all Borel sets in $(0,\infty]$. The second requirement for \eqref{e:pp.conv.for.main} is that
\begin{align}
v_n &=\max_{1\le \ell \le k+1} n^{2k+4-\ell} \int_{(\bbr^d)^{2k+4-\ell}}\one \Big\{ \sum_{i=1}^m a_i h_{r_ns_i}(x_1,\dots,x_{k+2}) h_{r_nt_i}(x_1,\dots,x_{k+2}) \neq 0 \Big\} \label{e:overlapping.prob} \\
&\times \one \Big\{ \sum_{i=1}^m a_i h_{r_ns_i}(x_1,\dots,x_\ell, x_{k+3}, \dots, x_{2k+4-\ell}) h_{r_nt_i}(x_1,\dots,x_\ell, x_{k+3}, \dots, x_{2k+4-\ell}) \neq 0 \Big\} \notag \\
&\qquad \qquad \qquad \qquad \qquad \qquad \qquad \qquad \qquad \qquad \qquad \qquad  \times \prod_{i=1}^{2k+4-\ell}f(x_i)d\bx \to 0, \ \ \ n\to\infty. \notag
\end{align}
For \eqref{e:total.variation}, we obtain that for every Borel set $A\subset (0,\infty)$,
\begin{align}
\E [\eta_{k,n}(A)] &= \binom{n}{k+2} \int_{(\bbr^d)^{k+2}} \one \Big\{ \sum_{i=1}^m a_i h_{r_ns_i}(x_1,\dots,x_{k+2}) h_{r_nt_i}(x_1,\dots,x_{k+2})\in A \Big\} \prod_{i=1}^{k+2} f(x_i) d\bx \label{e:one.side.unif.conv}\\
&=\binom{n}{k+2} r_n^{d(k+1)}\int_{\bbr^d}\int_{(\bbr^d)^{k+1}} \one \Big\{ \sum_{i=1}^m a_i h_{s_i}(0,\by) h_{t_i}(0,\by)\in A \Big\} \notag \\
&\qquad \qquad \qquad \qquad \qquad \qquad\qquad \qquad \qquad \times f(x) \prod_{i=1}^{k+1} f(x+r_ny_i)d\by dx \notag
\end{align}
with $\by=(y_1,\dots,y_{k+1})\in (\bbr^d)^{k+1}$. The second line above is obtained by changing the variables, $x_i=x+r_n y_{i-1}$, $i=1,\dots, k+2$ (with $y_0\equiv 0$) and the translation invariance of \eqref{e:def.hr}.
On the other hand, one can rewrite $\E[\eta_k(A)]$ as
$$
\E[\eta_k(A)]= \frac{1}{(k+2)!}\, \int_{\bbr^d}f(x)^{k+2}dx \int_{(\bbr^d)^{k+1}} \one \Big\{ \sum_{i=1}^m a_i h_{s_i}(0,\by) h_{t_i}(0,\by)\in A \Big\}  d\by.
$$
By the continuity of $f$, it holds that $\prod_{i=1}^{k+1}f(x+r_ny_i) \to f(x)^{k+1}$ as $n\to\infty$ a.e. Moreover, $\binom{n}{k+2} r_n^{d(k+1)}\to 1/(k+2)!$, $n\to\infty$; therefore, the uniform convergence in \eqref{e:total.variation} holds as $n\to\infty$.

For the proof of \eqref{e:overlapping.prob}, 
%check performing
the change of variables $x_i=x+r_ny_{i-1}$, $i=1,\dots,2k+4-\ell$ (with $y_0\equiv 0$) gives that
\begin{align*}
v_n &=\max_{1\le \ell \le k+1} n^{2k+4-\ell}r_n^{d(2k+3-\ell)} \int_{\bbr^d}\int_{(\bbr^d)^{2k+3-\ell}} \one \Big\{ \sum_{i=1}^m a_i h_{s_i}(0,y_1,\dots,y_{k+1}) h_{t_i}(0,y_1,\dots,y_{k+1})\neq 0 \Big\} \\
&\times \one \Big\{ \sum_{i=1}^m a_i h_{s_i}(0, y_1,\dots,y_{\ell-1}, y_{k+2}, \dots,y_{2k+3-\ell}) h_{t_i}(0, y_1,\dots,y_{\ell-1}, y_{k+2}, \dots,y_{2k+3-\ell})\neq 0 \Big\} \\
&\qquad \qquad \qquad \qquad \qquad \qquad \qquad \qquad \qquad \times f(x)\prod_{i=1}^{2k+3-\ell} f(x+r_ny_{i}) d\by dx \\
&\le \max_{1\le \ell \le k+1} n^{2k+4-\ell}r_n^{d(2k+3-\ell)}  \|f\|_\infty^{2k+3-\ell} \\
&\times \int_{(\bbr^d)^{2k+3-\ell}} \one \Big\{ \sum_{i=1}^m a_i h_{s_i}(0,y_1,\dots,y_{k+1}) h_{t_i}(0,y_1,\dots,y_{k+1})\neq 0 \Big\} \\
&\times \Big\{ \sum_{i=1}^m a_i h_{s_i}(0, y_1,\dots,y_{\ell-1}, y_{k+2}, \dots,y_{2k+3-\ell}) h_{t_i}(0, y_1,\dots,y_{\ell-1}, y_{k+2}, \dots,y_{2k+3-\ell})\neq 0 \Big\} d\by.
\end{align*}
The last expression vanishes as $n\to\infty$, because the integral is finite and
$$
\max_{1\le \ell \le k+1} n^{2k+4-\ell}r_n^{d(2k+3-\ell)}   = \max_{1\le \ell \le k+1} (nr_n^d)^{k+2-\ell} \to 0, \ \ \ n\to\infty.
$$
Now, we obtain \eqref{e:pp.conv.for.main}.

Subsequently, we claim that, as $n\to\infty$,
\begin{equation}  \label{e:first.term.asym}
\sum_{i=1}^m a_i G_{k,n}(s_i,t_i) \Rightarrow \sum_{i=1}^m a_i \int_{(\bbr^d)^{k+1}} h_{s_i}(0,\by)h_{t_i}(0,\by)M_k(d\by).
\end{equation}
Since the choice of $a_i$'s is arbitrary, \eqref{e:first.term.asym} is equivalent to
\begin{equation}  \label{e:first.term.asym2}
\big( G_{k,n}(s_i,t_i), \, i=1,\dots,m \big) \Rightarrow \Big( \int_{(\bbr^d)^{k+1}} h_{s_i}(0,\by) h_{t_i}(0,\by) M_k(d\by), \, i=1,\dots,m \Big).
\end{equation}
For the proof of \eqref{e:first.term.asym}, let $\gamma>0$ and define a continuous map $\widehat T_\gamma: M_p\big( (0,\infty] \big)\to \bbr_+$ by $\widehat T_\gamma \big( \sum_j \delta_{x_j} \big)=\sum_j x_j \one \{ x_j \ge \gamma \}$ (the continuity of $\widehat T_\gamma$ is proven in Section 7.2.3 of \cite{resnick:2007}). Applying the continuous mapping theorem to  \eqref{e:pp.conv.for.main}, we have
$$
\widehat T_\gamma(\eta_{k,n})\Rightarrow \widehat T_\gamma(\eta_k) \ \ \text{as } n\to\infty.
$$
Clearly, as $\gamma\downarrow0$,
$$
\widehat T_\gamma(\eta_k) =\sum_{i=1}^M Y_i \one \{ Y_i \ge \gamma \} \to \sum_{i=1}^M Y_i,  \ \ \ \text{a.s.}
$$
Furthermore,
$$
\sum_{i=1}^M Y_i \stackrel{d}{=}  \sum_{i=1}^m a_i \int_{(\bbr^d)^{k+1}} h_{s_i}(0,\by)h_{t_i}(0,\by) M_k (d\by).
$$
One can prove this by computing the Laplace transforms of both sides. On the one hand, Theorem 5.1 in \cite{resnick:2007} demonstrates that, for every $\lambda >0$,
\begin{align}
&\E \Big[ \exp \Big\{ -\lambda \sum_{i=1}^m a_i \int_{(\bbr^d)^{k+1}} h_{s_i}(0,\by)h_{t_i}(0,\by) M_k (d\by) \Big\} \Big] \notag\\
&= \exp \Big\{ -C_k \int_{(\bbr^d)^{k+1}} \Big( 1-e^{-\lambda \sum_{i=1}^m a_i h_{s_i}(0,\by)h_{t_i}(0,\by)} \Big)d\by \Big\}. \label{e:Laplace.transf}
\end{align}
On the other hand, it is a simple exercise to check that the Laplace transform of $\sum_{i=1}^MY_i$ is equal to \eqref{e:Laplace.transf}.

Now, \eqref{e:first.term.asym} follows if one can show that, for every $\epsilon>0$,
$$
\lim_{\gamma\downarrow0}\limsup_{n\to\infty} \P\bigg( \sum_{i=1}^m a_i \sum_{\substack{\Y\subset \X_n, \\ |\Y|=k+2}} h_{r_ns_i}(\Y)h_{r_n t_i}(\Y)\, \one \Big\{  \sum_{i=1}^m a_i h_{r_ns_i}(\Y)h_{r_n t_i}(\Y) \le \gamma \Big\} \ge \epsilon\bigg)=0.
$$
By Markov's inequality and the customary change of variables as in \eqref{e:one.side.unif.conv}, as well as $n^{k+2}r_n^{d(k+1)}=1$, we have that
\begin{align*}
&\P\bigg( \sum_{i=1}^m a_i \sum_{\substack{\Y\subset \X_n, \\ |\Y|=k+2}} h_{r_ns_i}(\Y)h_{r_n t_i}(\Y)\, \one \Big\{  \sum_{i=1}^m a_i  h_{r_ns_i}(\Y)h_{r_n t_i}(\Y) \le \gamma \Big\} \ge \epsilon\bigg) \\
&\le \frac{1}{\epsilon}\sum_{i=1}^m a_i \binom{n}{k+2} \int_{(\bbr^d)^{k+2}} h_{r_ns_i}(x_1,\dots,x_{k+2})h_{r_n t_i}(x_1,\dots,x_{k+2}) \\
&\qquad \qquad \qquad\qquad \qquad \times \one \Big\{ \sum_{i=1}^m a_i h_{r_ns_i}(x_1,\dots,x_{k+2})h_{r_n t_i}(x_1,\dots,x_{k+2}) \le \gamma \Big\} \prod_{i=1}^{k+2}f(x_i)d\bx \\
&\le \frac{\|f\|_\infty^{k+1}}{\epsilon(k+2)!}\sum_{i=1}^m a_i \int_{(\bbr^d)^{k+1}} h_{s_i}(0,\by)h_{t_i}(0,\by) \one \Big\{ \sum_{i=1}^m a_i h_{s_i}(0,\by)h_{t_i}(0,\by) \le \gamma \Big\}d\by.
\end{align*}
The last term converges to $0$ as $\gamma \downarrow 0$ by the dominated convergence theorem, and we have obtained \eqref{e:first.term.asym2}.

From the argument thus far, the entire proof can be completed, provided that, for all $0\le s \le t \le \infty$,
\begin{equation}  \label{e:Betti.conv.in.prob}
\betakn(s,t) - G_{k,n}(s,t) \stackrel{p}{\to} 0, \ \ \ n\to\infty.
\end{equation}
It follows from Lemma \ref{l:Betti.bounds} that
\begin{align}
&\big| \, \betakn(s,t) - G_{k,n}(s,t) \, \big|  \label{e:diff.Betti.Gkn} \\
&\le G_{k,n}(s,t)  - \sum_{\Y\subset \X_n, \, |\Y|=k+2} g_{r_ns, r_nt} (\Y, \X_n) + \binom{k+3}{k+1} L_{r_nt}  \notag \\
&\le \sum_{\Y\subset \X_n, \, |\Y|=k+2} \hspace{-10pt} \one \big\{ \C(\Y, r_nt) \text{ is connected} \big\} \notag \\
&\qquad \qquad \quad  \times \one \big\{ \| y-z \| \le r_nt \text{ for some } y\in \Y \text{ and } z\in \X_n \setminus \Y \big\} + \binom{k+3}{k+1} L_{r_nt} \notag \\
&\le (k+3)L_{r_nt} + \binom{k+3}{k+1} L_{r_nt}. \notag
\end{align}
By \eqref{e:conv.Lr} in Lemma \ref{l:Betti.bounds}, we have that $\E[L_{r_nt}]\to $ as $n\to\infty$, and therefore, \eqref{e:Betti.conv.in.prob} follows as required.

\subsection{Proof of Theorem \ref{t:divergence.regime}}  \label{sec:proof.divergence.regime}
In the course of our proof, $C>0$ denotes a generic positive constant, which is independent of $n$ but may change line by line.
For the brevity of the proof, some of the technical results and arguments have been moved to the Appendix.

It is known from Corollary A.3 in \cite{hiraoka:shirai:trinh:2018} that for every $\mu\in M_+(\DelL)$, $\A$ contains a countable \emph{convergence-determining class} for $\mu$; that is, there exists a countable subset $\A' \subset \A$ such that $\mu (\partial A)=0$ for every $A\in \A'$, and for any $(\mu_n)\subset M_+(\DelL)$,
$$
\mu_n(A)\to\mu(A) \ \ \text{for all } A\in \A',
$$
implies $\mu_n \stackrel{v}{\to}\mu$ in $M_+(\DelL)$. Combining these properties with Proposition 3.4 in \cite{hiraoka:shirai:trinh:2018}, one can immediately obtain \eqref{e:vague.as.conv} and \eqref{e:exp.vague.conv} as a direct consequence of
$$
\frac{\xi_{k,n}(\genR)}{\seq}  \to C_k \int_{(\bbr^d)^{k+1}}\genH (0,\by) d\by, \ \ \ n\to\infty, \ \ \text{a.s.},
$$
and
$$
\frac{\E\big[\xi_{k,n}(\genR)\big]}{\seq}  \to C_k \int_{(\bbr^d)^{k+1}}\genH (0,\by) d\by, \ \ \ n\to\infty,
$$
for every $\genR\in \A$, where $\by=(y_1,\dots,y_{k+1})\in (\bbr^d)^{k+1}$. 

Since \eqref{e:PD-persistent.Betti} implies that the marginal distributions of $\xi_{k,n}$ are expressed as a linear combination of the persistent Betti numbers, it suffices to show that, for all $0 \le s \le t \le \infty$,
\begin{equation}  \label{e:SLLN.Betti}
\frac{\betakn(s,t)}{\seq} \to A_k(s,t), \ \ \ n\to\infty, \ \ \text{a.s.},
\end{equation}
and
\begin{equation}  \label{e:exp.Betti.conv}
\frac{\E\big[\betakn(s,t)\big]}{\seq} \to A_k(s,t), \ \ \ n\to\infty,
\end{equation}
where
\begin{equation}  \label{e:def.Akst}
A_k(s,t) = C_k \int_{(\bbr^d)^{k+1}} h_s(0,\by)h_t(0,\by)d\by.
\end{equation}
For the proof of \eqref{e:SLLN.Betti}, we first establish the 
SLLN for $G_{k,n}(s,t)$, i.e.,
\begin{equation}  \label{e:SLLN.Gkn}
\frac{G_{k,n}(s,t)}{\seq} \to A_k(s,t), \ \ \ n\to\infty, \ \ \text{a.s.},
\end{equation}
where  $G_{k,n}(s,t)$ is defined at \eqref{e:def.Gknst}.
Subsequently, it is demonstrated that the difference between $G_{k,n}(s,t)$ and $\betakn(s,t)$ with both scaled by $\seq$ is almost surely negligible as $n\to\infty$.  
In order to handle the potentially  slow growth rate (e.g., logarithmic growth rate) of the scaler for $G_{k,n}(s,t)$, our idea is to employ the concentration bounds given in \cite{bachmann:reitzner:2018} (see Proposition \ref{p:concentration} of the Appendix). The other key machinery for our proof is to detect some subsequential bounds for $G_{k,n}(s,t)$. 
%Additionally, for the application of these concentration inequalities one needs to detect nice subsequential upper and lower bounds for $G_{k,n}(s,t)$. The latter is  a standard technique for the theory of geometric graphs; see, e.g., Chapter 3 of the monograph \cite{penrose:2003}. Since a \v{C}ech complex is a higher-dimensional analogue of a geometric graph (the latter is in fact a one-dimensional skeleton of the former), it is possible to extend this argument to our higher-dimensional setup.
%Before starting the proof of \eqref{e:SLLN.Gkn}, 

Now, we begin with some preliminary work. For each $n\ge 1$, let $N_n$ be a Poisson random variable with mean $n$. Assume that $N_n$ is independent of $(X_i)$. We then define
$$
\Pn=\begin{cases}
\{X_1,\dots,X_{N_n}\}, &  N_n \ge 1, \\
\emptyset, & N_n = 0,
\end{cases}
$$
to be a Poisson point process with intensity $nf$. Additionally, define $v_m=\lfloor e^{m^\gamma}\rfloor$, $m=0,1,2,\dots$ for some $\gamma\in (0,1)$, and let
\begin{equation}  \label{e:def.pm.qm}
p_m = \text{argmax}\{ v_m \le \ell \le v_{m+1}: r_\ell \}, \ \ \  q_m = \text{argmin}\{ v_m \le \ell \le v_{m+1}: r_\ell \}.
\end{equation}
Then, for every $n\ge 1$, there exists a unique $m = m(n)\in \bbn$ such that $v_m \le n < v_{m+1}$.
Moreover, we decompose the indicator \eqref{e:def.hr} by
\begin{align}
h_r(x_1,\dots,x_{k+2}) &= \one \Big\{  \bigcap_{j=1, \, j \neq j_0}^{k+2}B(x_j, r/2) \neq \emptyset \text{ for all } j_0 \in \{ 1,\dots,k+2 \} \Big\}  - \one  \Big\{  \bigcap_{j=1}^{k+2}B(x_j, r/2) \neq \emptyset \Big\} \notag \\
&=: h_r^{(+)}(x_1,\dots,x_{k+2}) - h_r^{(-)}(x_1,\dots,x_{k+2}).  \label{e:def.hrpm}
\end{align}
Then, $h_r^{(\pm)}(x_1,\dots, x_{k+2})$ are both non-decreasing in $r$ for all $x_i$'s, $i=1,\dots,m$. Using these indicator functions, $G_k(s,t)$ can be split into four terms:
\begin{align*}
G_{k,n}(s,t) &= \sum_{\Y\subset \X_n, \, |\Y|=k+2} \big[ h_{r_ns}^{(+)}(\Y)h_{r_nt}^{(+)}(\Y) - h_{r_ns}^{(+)}(\Y)h_{r_nt}^{(-)}(\Y) - h_{r_ns}^{(-)}(\Y)h_{r_nt}^{(+)}(\Y) + h_{r_ns}^{(-)}(\Y)h_{r_nt}^{(-)}(\Y)\big] \\
&=: G_{k,n}^{(+,+)}(s,t) - G_{k,n}^{(+,-)}(s,t) - G_{k,n}^{(-,+)}(s,t) + G_{k,n}^{(-,-)}(s,t).
\end{align*}
From this decomposition, it is sufficient to prove that
$$
\frac{G_{k,n}^{(\pm,\pm)}(s,t)}{\seq} \to A_k^{(\pm,\pm)}(s,t):=C_k \int_{(\bbr^d)^{k+1}} h_s^{(\pm)}(0,\by)h_t^{(\pm)}(0,\by)d\by, \ \ \ n\to\infty, \ \ \text{a.s.}
$$
In order to avoid repetition of the same arguments, we prove the SLLN only for $G_{k,n}^{(+,+)}(s,t)$. For ease of notation, we henceforth drop the superscript $(+,+)$ from $G_{k,n}^{(+,+)}(s,t)$ and $A_k^{(+,+)}(s,t)$, while omitting $(+)$ from the indicators in the  limit.

By the monotonicity of indicators in \eqref{e:def.hrpm}, we can bound $G_{k,n}(s,t)$ by
\begin{align*}
G_{k,n}(s,t) &\le \sum_{\Y\subset \X_{v_{m+1}}, \, |\Y|=k+2} h_{r_{p_m}s}(\Y)h_{r_{p_m}t}(\Y).
\end{align*}
Noting that $\seq \ge v_m^{k+2}r_{q_m}^{d(k+1)}$, we have for every $n\ge 1$,
\begin{equation}  \label{e:upper.bound.Gkn}
\frac{G_{k,n}(s,t)}{\seq} \le \frac{T_m^{\uparrow}(s,t)}{v_m^{k+2}r_{q_m}^{d(k+1)}} + \big(v_m^{k+2}r_{q_m}^{d(k+1)}\big)^{-1} \Big[  \sum_{\Y\subset \X_{v_{m+1}}, \, |\Y|=k+2} h_{r_{p_m}s}(\Y)h_{r_{p_m}t}(\Y) - T_m^{\uparrow}(s,t) \Big],
\end{equation}
where
\begin{equation}  \label{e:def.Tm.uparrow}
T_m^{\uparrow}(s,t) = \sum_{\Y\subset \mathcal P_{v_{m+1}}, \, |\Y|=k+2} h_{r_{p_m}s}(\Y)h_{r_{p_m}t}(\Y).
\end{equation}

Similarly, $G_{k,n}(s,t)$ can be bounded below by
\begin{equation}  \label{e:lower.bound.Gkn}
\frac{G_{k,n}(s,t)}{\seq} \ge \frac{T_m^{\downarrow}(s,t)}{v_{m+1}^{k+2}r_{p_m}^{d(k+1)}} + \big(v_{m+1}^{k+2}r_{p_m}^{d(k+1)}\big)^{-1} \Big[  \sum_{\Y\subset \X_{v_{m}}, \, |\Y|=k+2} h_{r_{q_m}s}(\Y)h_{r_{q_m}t}(\Y) - T_m^{\downarrow}(s,t) \Big],
\end{equation}
where
\begin{equation}  \label{e:def.Tm.downarrow}
T_m^{\downarrow}(s,t) = \sum_{\Y\subset \mathcal P_{v_{m}}, \, |\Y|=k+2} h_{r_{q_m}s}(\Y)h_{r_{q_m}t}(\Y).
\end{equation}
It follows from \eqref{e:upper.bound.Gkn}, \eqref{e:lower.bound.Gkn}, and Lemma \ref{l:diff.Poisson.binomial} in the Appendix that
\begin{align*}
\liminf_{m\to\infty} \frac{T_m^{\downarrow}(s,t)}{v_{m+1}^{k+2}r_{p_m}^{d(k+1)}} &\le \liminf_{n\to\infty} \frac{G_{k,n}(s,t)}{\seq} \le \limsup_{n\to\infty} \frac{G_{k,n}(s,t)}{\seq} \le \limsup_{m\to\infty} \frac{T_m^{\uparrow}(s,t)}{v_m^{k+2}r_{q_m}^{d(k+1)}}, \ \ \text{a.s.}
\end{align*}
Therefore, the required SLLN for $G_{k,n}(s,t)$ will follow, if we can prove that for every $\epsilon>0$
\begin{align*}
&\limsup_{m\to\infty} \frac{T_m^{\uparrow}(s,t)}{v_m^{k+2}r_{q_m}^{d(k+1)}} \le (1+\epsilon)^2 A_k(s,t), \ \ \text{a.s.}, \\
&\liminf_{m\to\infty} \frac{T_m^{\downarrow}(s,t)}{v_{m+1}^{k+2}r_{p_m}^{d(k+1)}} \ge (1+\epsilon)^{-2} A_k(s,t), \ \ \text{a.s.}
\end{align*}
By the Borel--Cantelli lemma, it suffices to demonstrate that
\begin{align}
&\sum_{m=1}^\infty \P \big( T_m^{\uparrow}(s,t) >(1+\epsilon)^2 A_k(s,t) v_m^{k+2} r_{q_m}^{d(k+1)} \big) < \infty, \label{e:BC.lemma1}\\
&\sum_{m=1}^\infty \P \big( T_m^{\downarrow}(s,t) <(1+\epsilon)^{-2} A_k(s,t) v_{m+1}^{k+2} r_{p_m}^{d(k+1)} \big) < \infty. \label{e:BC.lemma2}
\end{align}
By virtue of Lemma \ref{l:moments} in the Appendix, there exists $N\in\bbn$ such that for all $m\ge N$
$$
\E\big[ T_m^{\uparrow}(s,t) \big] \le (1+\epsilon)A_k(s,t)v_m^{k+2}r_{q_m}^{d(k+1)}.
$$
To utilize the concentration inequality in Proposition \ref{p:concentration} of the Appendix, we note that $\mathcal P_{v_{m+1}}$ is a Poisson point process with finite intensity measure $v_{m+1}f$. Furthermore, $T_m^\uparrow(s,t)$ is a Poisson $U$-statistics of order $k+2$, which satisfies \eqref{e:cond.s1} and \eqref{e:cond.s2}.
Now, we can apply \eqref{e:concentration1} to obtain that
\begin{align*}
&\P \big( T_m^{\uparrow}(s,t) >(1+\epsilon)^2 A_k(s,t) v_m^{k+2} r_{q_m}^{d(k+1)} \big) \\
&\le \P \Big( T_m^{\uparrow}(s,t) - \E\big[T_m^{\uparrow}(s,t)\big] >\epsilon(1+\epsilon) A_k(s,t) v_m^{k+2} r_{q_m}^{d(k+1)} \Big) \\
&\le \exp \bigg\{  -C \Big[  \Big(  \E[T_m^\uparrow (s,t)] + \epsilon(1+\epsilon) A_k(s,t) v_m^{k+2} r_{q_m}^{d(k+1)} \Big)^{1/(2k+4)}  - \Big(  \E[T_m^\uparrow (s,t)] \Big)^{1/(2k+4)} \Big]^2 \bigg\} \\
&\le \exp \Big\{ -C(v_m^{k+2}r_{q_m}^{d(k+1)})^{1/(k+2)} \Big\},
\end{align*}
where the last inequality follows from \eqref{e:moment.Tm.uparrow} in Lemma \ref{l:moments}.
Subsequently, applying Lemma \ref{l:RV.seq} and \eqref{e:faster.than.log}, we have that
$$
\exp \Big\{ -C(v_m^{k+2}r_{q_m}^{d(k+1)})^{1/(k+2)} \Big\} \le \exp \Big\{ -C(v_m^{k+2}r_{v_m}^{d(k+1)})^{1/(k+2)} \Big\} \le \exp\big\{   -Cm^{\gamma \eta /(k+2)}\big\}.
$$
Clearly, the last term is summable with respect to $m$, and therefore, we have obtained \eqref{e:BC.lemma1}.

Turning to condition \eqref{e:BC.lemma2}, note that $T_m^\downarrow(s,t)$ is again a Poisson $U$-statistics of order $k+2$.
Lemma \ref{l:moments} demonstrates that there exists $N'\in\bbn$, so that for all $m\ge N'$,
$$
\E\big[ T_m^\downarrow (s,t) \big] \ge (1+\epsilon)^{-1} A_k(s,t) v_{m+1}^{k+2} r_{p_m}^{d(k+1)},
$$
and, by \eqref{e:concentration2} in Proposition \ref{p:concentration},
\begin{align*}
&\P \big( T_m^{\downarrow}(s,t) <(1+\epsilon)^{-2} A_k(s,t) v_{m+1}^{k+2} r_{p_m}^{d(k+1)} \big) \\
&\le \P \Big( T_m^{\downarrow}(s,t) -\E\big[ T_m^{\downarrow}(s,t) \big] < -\epsilon(1+\epsilon)^{-2} A_k(s,t)   v_{m+1}^{k+2} r_{p_m}^{d(k+1)} \Big)\\
&\le \exp \bigg\{ - \frac{C\big( \epsilon(1+\epsilon)^{-2}A_k(s,t) v_{m+1}^{k+2}r_{p_m}^{d(k+1)} \big)^2}{\text{Var} \big( T_m^\downarrow(s,t) \big)} \bigg\}. 
\end{align*}
%According to , Var$\big( T_m^\downarrow(s,t) \big)$ grows at the same rate as $\E[T_m^\downarrow(s,t)]$, that is, as $m\to\infty$,
%\begin{equation} \label{e:growth.rate.variance}
%\text{Var} \big( T_m^\downarrow(s,t) \big)\sim A_k(s,t)v_{m+1}^{k+2}r_{p_m}^{d(k+1)}.
%\end{equation}
Combining \eqref{e:2nd.moment.Tm.downarrow} in Lemma \ref{l:moments}, Lemma \ref{l:RV.seq}, and \eqref{e:faster.than.log}, we have that
$$
\frac{\big( v_{m+1}^{k+2}r_{p_m}^{d(k+1)} \big)^2}{\text{Var}\big( T_m^\downarrow(s,t) \big)} \ge C v_{m+1}^{k+2}r_{p_m}^{d(k+1)} \ge C v_m^{k+2}r_{v_m}^{d(k+1)} \ge Cm^{\gamma \eta}.
$$
This concludes that
$$
\P \big( T_m^{\downarrow}(s,t) <(1+\epsilon)^{-2} A_k(s,t) v_{m+1}^{k+2} r_{p_m}^{d(k+1)} \big) \le \exp \{ -Cm^{\gamma \eta} \}.
$$
The right hand side above is summable with respect to $m$, and therefore, \eqref{e:BC.lemma2} has been established.

Having obtained \eqref{e:SLLN.Gkn}, our next goal is to show that
\begin{equation}  \label{e:difference.vanish.as}
(\seq)^{-1} \big|\, \betakn(s,t)-G_{k,n}(s,t)\, \big| \to 0, \ \ n\to\infty, \ \ \text{a.s.}
\end{equation}
It follows from \eqref{e:diff.Betti.Gkn} that one can bound \eqref{e:difference.vanish.as} by a constant multiple of $(\seq)^{-1}L_{r_nt}$, which itself is further bounded as follows.
\begin{align}
\frac{L_{r_nt}}{\seq} &\le \frac{1}{v_m^{k+2} r_{q_m}^{d(k+1)}} \sum_{\substack{\Y \subset \X_{v_{m+1}}, \\ |\Y|=k+3}} \one \big\{ \C(\Y, r_{p_m}t) \text{ is connected} \big\}  \notag\\
&= \frac{V_{m}(t)}{v_m^{k+2} r_{q_m}^{d(k+1)}} + \big( v_m^{k+2} r_{q_m}^{d(k+1)} \big)^{-1} \Big\{ \sum_{\substack{\Y \subset \X_{v_{m+1}}, \\ |\Y|=k+3}} \one \big\{ \C(\Y, r_{p_m}t) \text{ is connected} \big\} -V_{m}(t) \Big\}, \label{e:Vmk.vanish}
\end{align}
where
\begin{equation}  \label{e:def.Vm}
V_{m}(t) = \sum_{\substack{\Y \subset \mathcal P_{v_{m+1}}, \\ |\Y|=k+3}} \one \big\{ \C(\Y, r_{p_m}t) \text{ is connected} \big\}.
\end{equation}
For the first inequality in \eqref{e:Vmk.vanish}, we have used 
the fact that $\one \big\{  \C(\Y,r) \text{ is connected} \big\}$ is non-decreasing in $r$.
Because of Lemma \ref{l:diff.Poisson.binomial}, it now remains to show that
\begin{equation}  \label{e:Vmk.as.0}
\frac{V_{m}(t)}{v_m^{k+2} r_{q_m}^{d(k+1)}} \to 0, \ \ \ m\to\infty, \ \ \text{a.s.}
\end{equation}
By \eqref{e:moment.Vm} in Lemma \ref{l:moments}, Lemma \ref{l:RV.seq}, and $nr_n^d\to 0$, $n\to\infty$,
\begin{align}
\E \big[ V_{m}(t)\big] &\sim \frac{v_{m}^{k+3} r_{q_m}^{d(k+2)}}{(k+3)!}\, \int_{\bbr^d}f(x)^{k+3}dx\int_{(\bbr^d)^{k+2}} \one \big\{ \C(\{ 0,\by \}, t) \text{ is connected} \big\}d\by  \label{e:asym.EVmk} \\
&= o \big( v_m^{k+2} r_{q_m}^{d(k+1)} \big), \ \ \ m\to\infty,  \notag
\end{align}
where $\by=(y_1,\dots,y_{k+2})\in (\bbr^d)^{k+2}$. 
Since $V_m(t)$ is a Poisson $U$-statistics of order $k+3$ satisfying \eqref{e:cond.s1} and \eqref{e:cond.s2}, one can again employ the concentration inequality in Proposition \ref{p:concentration}. Indeed, for every $\epsilon>0$ and sufficiently large $m$, we have that
\begin{align*}
&\P \big( V_{m}(t) > \epsilon v_m^{k+2} r_{q_m}^{d(k+1)} \big) \\
&\le \P \Big( V_{m}(t) - \E\big[ V_{m}(t) \big] > \frac{\epsilon}{2}\, v_m^{k+2} r_{q_m}^{d(k+1)}  \Big) \\
&\le \exp \bigg\{ -C\Big[ \Big( \E\big[V_{m}(t) \big] + \frac{\epsilon}{2}\, v_m^{k+2} r_{q_m}^{d(k+1)}  \Big)^{1/(2k+6)} - \Big( \E\big[ V_{m}(t) \big] \Big)^{1/(2k+6)} \Big]^2 \bigg\}.
\end{align*}
By \eqref{e:asym.EVmk}, Lemma \ref{l:RV.seq}, and \eqref{e:faster.than.log},
\begin{align*}
&\bigg[ \Big( \E\big[V_{m}(t) \big] + \frac{\epsilon}{2}\, v_m^{k+2} r_{q_m}^{d(k+1)}  \Big)^{1/(2k+6)} - \Big( \E\big[ V_{m}(t) \big] \Big)^{1/(2k+6)} \bigg]^2  \\
&\ge C\big(  v_m^{k+2} r_{q_m}^{d(k+1)}  \big)^{1/(k+3)} \ge C \big(  v_m^{k+2} r_{v_m}^{d(k+1)}  \big)^{1/(k+3)} \ge Cm^{\gamma \eta/(k+3)}.
\end{align*}
Since $\sum_m e^{-Cm^{\gamma \eta /(k+3)}} < \infty$, the Borel--Cantelli lemma completes the proof of \eqref{e:Vmk.as.0}. Now, we have obtained \eqref{e:SLLN.Betti}.

Finally, we move on to the proof of \eqref{e:exp.Betti.conv}. By \eqref{e:diff.Betti.Gkn} and \eqref{e:conv.Lr} in Lemma \ref{l:Betti.bounds}, it suffices to show that
$$
\frac{\E \big[ G_{k,n}(s,t) \big]}{\seq} \to A_k(s,t),  \ \ \text{as } n\to\infty.
$$
This is, however, not difficult to prove by the customary change of variables, together with the dominated convergence theorem.

\subsection{Proof of Theorem \ref{t:vanishing.regime}}  \label{sec:proof.vanishing}

Given $H_1, H_2\in C_K^+(\DelL)$ and $\epsilon_1, \epsilon_2>0$, define $F_{H_1,H_2,\epsilon_1,\epsilon_2}: M_p(\DelL)\to [0,\infty)$ by
$$
F_{H_1,H_2,\epsilon_1,\epsilon_2}(\xi) = \big( 1-e^{-(\xi(H_1)-\epsilon_1)_+} \big)\big( 1-e^{-(\xi(H_2)-\epsilon_2)_+} \big),
$$
where $\xi(H_i) = \int_{\DelL} H_i(x)\xi(dx)$ and
$(a)_+=a$ if $a\ge 0$ and $0$ otherwise. It is straightforward to check that $F_{H_1,H_2,\epsilon_1,\epsilon_2}\in \mathcal C_0$. Denote
\begin{align*}
\eta_n(\cdot) &:= (n^{k+2}r_n^{d(k+1)})^{-1} \P (\xi_{k,n}\in \cdot), \\
\eta(\cdot) &:= C_k \lambda \Big\{ \by\in (\bbr^d)^{k+1}: \delta_{( b(0,\by), d(0,\by))} \in \cdot \Big\}.
\end{align*}
According to Theorem A.2 in \cite{hult:samorodnitsky:2010}, the proof will be complete if one can show that
$$
\eta_n(F_{H_1,H_2,\epsilon_1,\epsilon_2}) \to \eta(F_{H_1,H_2,\epsilon_1,\epsilon_2}), \ \ \ n\to\infty,
$$
for all $H_1, H_2 \in C_K^+(\DelL)$ and $\epsilon_1, \epsilon_2 > 0$. Let
$$
\Delta_n = \prod_{\ell=1}^2 \bigg( 1-\exp \Big\{ -  \Big( \sum_{(b_i,d_i)\in D_{k,n}} H_\ell(b_i, d_i) - \epsilon_\ell\Big)_+ \Big\}\bigg);
$$
then,
$$
\eta_n(F_{H_1,H_2,\epsilon_1,\epsilon_2}) = (\seq)^{-1} \E[\Delta_n].
$$
For each $\ell=1,2$, $H_\ell$ has compact support in $\DelL$; hence, there exists $0<T<\infty$ such that
\begin{equation}  \label{e:support.cond}
\bigcup_{\ell=1}^2 \text{supp}(H_\ell) \subset \big( [0,T]\times (0,\infty] \big) \cap \DelL
\end{equation}
(supp$(H_\ell)$ denotes the support of $H_\ell$).

For ease of the description below, let us introduce some shorthand notations. Denote a collection of ordered $m$-tuples by
$$
\I_m := \big\{\bi=(i_1,\dots,i_m)\in \bbn_+^m: 1\le i_1 < \cdots <i_m \le n  \big\}.
$$
Write also
$$
\X_\bi = (X_{i_1}, \dots, X_{i_m}),  \ \ \ \bi = (i_1,\dots,i_m)\in \I_m.
$$
Then,
$$
\eta_n(F_{H_1,H_2,\epsilon_1,\epsilon_2}) = (\seq)^{-1} \E[\Delta_n \one_{U_n}] + (\seq)^{-1} \E[\Delta_n \one_{U_n^c}] := A_n + B_n,
$$
where
$$
U_n = \bigcap_{\bi\in \I_{k+3}} \big\{  \C(\X_\bi, r_nT) \text{ is not connected} \big\}.
$$
We first show that $B_n$ tends to $0$ as $n\to\infty$. Since $\Delta_n\le1$,
\begin{align*}
B_n &\le (\seq)^{-1} \P \Big( \bigcup_{\bi \in \I_{k+3}} \big\{ \C(\X_\bi, r_nT) \text{ is connected} \big\} \Big) \\
&\le (\seq)^{-1} \binom{n}{k+3}\int_{(\bbr^d)^{k+3}} \one \Big\{ \C\big(\{  x_1,\dots,x_{k+3} \}, r_nT\big) \text{ is connected} \Big\} \prod_{i=1}^{k+3}f(x_i) d\bx.
\end{align*}
By the change of variables, $x_i=x+r_ny_{i-1}$, $i=1,\dots,k+3$ (with $y_0\equiv 0$), the translation invariance of \eqref{e:def.hr}, and $\|f\|_\infty <\infty$,
\begin{align}
B_n &\le  (\seq)^{-1} \binom{n}{k+3}r_n^{d(k+2)} \int_{\bbr^d} \int_{(\bbr^d)^{k+2}} \one \Big\{ \C\big(\{  0,\by \}, T\big) \text{ is connected} \Big\} \label{e:change.variable.An}  \\
&\qquad \qquad \qquad \qquad \qquad \qquad\qquad \qquad \qquad \times f(x)\prod_{i=1}^{k+2}f(x+r_ny_i)d\by dx \notag \\
&\le \frac{\|f\|_\infty^{k+2}}{(k+3)!}\, nr_n^d\int_{(\bbr^d)^{k+2}} \one \Big\{ \C\big(\{  0,\by \}, T\big) \text{ is connected} \Big\} d\by \to 0, \ \ \ n\to\infty, \notag
\end{align}
where $\by=(y_1,\dots,y_{k+2})\in (\bbr^d)^{k+2}$. 
For the asymptotics of $A_n$, one can rewrite it as
\begin{equation}  \label{e:revised.An}
A_n = (\seq)^{-1}\E\bigg[ \Delta_n \one \Big\{ U_n \cap \bigcup_{\bi \in \Ik} \big\{ \sup_{0\le t \le T} h_{r_nt} (\X_\bi)=1 \big\} \Big\}  \bigg].
\end{equation}
To see this, suppose that $\supt h_{r_nt} (\X_\bi)=0$ for all $\bi\in \Ik$ and $\C(\X_\bi, r_nT)$ is disconnected for every $\bi\in \I_{k+3}$; then, it must be that the point process $\sum_{(b_i,d_i)\in D_{k,n}, \, b_i \le T}\delta_{(b_i,d_i)}$ becomes the null measure. Therefore, by \eqref{e:support.cond}, $\sum_{(b_i,d_i)\in D_{k,n}} H_\ell(b_i,d_i) = 0$ for each $\ell=1,2$, and thus, $\Delta_n=0$.

Next, we divide a newly added portion in \eqref{e:revised.An} into two parts:
\begin{align*}
\bigcup_{\bi \in \Ik} \big\{ \sup_{0\le t \le T} h_{r_nt} (\X_\bi)=1 \big\} &= \bigcup_{\bi\in \Ik} \big\{  \sup_{0\le t \le T} h_{r_nt} (\X_\bi) = 1, \\
&\qquad \qquad \qquad \supt h_{r_nt}(\X_{\bj})=0 \text{ for all } \bj \in \Ik \text{ with } \bj\neq \bi \big\} \\
&\qquad \cup \bigcup_{\bi\in \Ik} \bigcup_{\bj\in \Ik, \, \bj\neq \bi} \big\{ \sup_{0\le t \le T} h_{r_nt} (\X_\bi) = \sup_{0\le t \le T} h_{r_nt} (\X_{\bj}) = 1 \big\} \\
&=: V_n \cup W_n.
\end{align*}
Note that $V_n \cap W_n =\emptyset$. Accordingly, we can write
\begin{equation}  \label{e:Cn.Dn}
A_n = (\seq)^{-1} \E\big[ \Delta_n \one_{U_n \cap V_n} \big] + (\seq)^{-1} \E\big[ \Delta_n \one_{U_n \cap W_n} \big] =: C_n + D_n.
\end{equation}

In the following, we show that $D_n$ is negligible as $n\to\infty$. In fact,
\begin{align*}
D_n &\le (\seq)^{-1} \P(W_n) \\
&\le (\seq)^{-1} \sum_{\ell=0}^{k+1} \sum_{\bi\in \Ik} \sum_{\bj\in \Ik, \, |\bi\cap \bj|=\ell} \P \Big( \supt h_{r_nt} (\X_\bi) = \supt h_{r_nt} (\X_\bj) =1 \Big) \\
&= (\seq)^{-1} \binom{n}{k+2}\binom{n-k-2}{k+2} \P \Big( \supt h_{r_nt} (X_1,\dots,X_{k+2})=1 \Big)^2 \\
&+ (\seq)^{-1} \sum_{\ell=1}^{k+1}\binom{n}{k+2} \binom{k+2}{\ell} \binom{n-k-2}{k+2-\ell} \\
&\qquad \qquad \quad  \times\P\Big(  \supt h_{r_nt} (X_1,\dots,X_{k+2}) =\supt h_{r_nt}(X_1,\dots,X_\ell, X_{k+3}, \dots, X_{2k+4-\ell})=1 \Big) \\
&\le  (\seq)^{-1}n^{2k+4} \Big( \int_{(\bbr^d)^{k+2}} \supt h_{r_nt}(x_1,\dots,x_{k+2})\prod_{i=1}^{k+2}f(x_i)d\bx \Big)^2 \\
&+ \sum_{\ell=1}^{k+1} (\seq)^{-1} n^{2k+4-\ell} \int_{(\bbr^d)^{2k+4-\ell}} \supt h_{r_nt}(x_1,\dots,x_{k+2}) \\
&\qquad \qquad \qquad \qquad\qquad \qquad \qquad  \quad \times \supt h_{r_nt} (x_1,\dots,x_\ell, x_{k+3}, \dots,x_{2k+4-\ell}) \prod_{i=1}^{2k+4-\ell}f(x_i)d\bx \\
&=: E_n+F_n.
\end{align*}
Performing the same change of variables as in \eqref{e:change.variable.An},
\begin{align*}
E_n &= (\seq)^{-1} n^{2k+4} \Big( r_n^{d(k+1)} \int_{\bbr^d} \int_{(\bbr^d)^{k+1}} \supt h_t(0,\by) f(x)\prod_{i=1}^{k+1} f(x+r_ny_i) d\by dx \Big)^2 \\
&\le \seq \| f \|_\infty^{2k+2} \Big( \int_{(\bbr^d)^{k+1}} \supt h_t(0,\by)d\by \Big)^2 \to 0,  \ \ \text{as } n\to\infty, 
\end{align*}
where $\by=(y_1,\dots,y_{k+1})\in (\bbr^d)^{k+1}$. 
Similarly,
\begin{align*}
F_n &= \sum_{\ell=1}^{k+1} (\seq)^{-1} n^{2k+4-\ell} r_n^{d(2k+3-\ell)} \int_{\bbr^d} \int_{(\bbr^d)^{2k+3-\ell}} \supt h_t (0,y_1,\dots,y_{k+1}) \\
&\qquad \qquad \times \supt h_t (0,y_1,\dots,y_{\ell-1}, y_{k+2}, \dots,y_{2k+3-\ell}) f(x)\prod_{i=1}^{2k+3-\ell} f(x+r_n y_i) d\by dx \\
&\le \sum_{\ell=1}^{k+1} \| f\|_\infty ^{2k+3-\ell} (nr_n^d)^{k+2-\ell}  \int_{(\bbr^d)^{2k+3-\ell}} \supt h_t (0,y_1,\dots,y_{k+1}) \\
&\qquad \qquad \qquad \qquad\qquad \times \supt h_t (0,y_1,\dots,y_{\ell-1}, y_{k+2}, \dots,y_{2k+3-\ell}) d\by\to 0, \ \ \text{as } n\to\infty.
\end{align*}
Hence, we have established $D_n\to0$ as $n\to\infty$. \\
Returning to \eqref{e:Cn.Dn} and noting that
$$
\bigg( \Big\{ \supt h_{r_nt}(\X_\bi) = 1, \ \supt h_{r_nt} (\X_{\bj}) = 0 \text{ for all } \bj\in \Ik \text{ with } \bj\neq \bi  \Big\}, \, \bi \in \Ik \bigg)
$$
are disjoint, we have
\begin{align}
C_n &= (\seq)^{-1} \sum_{\bi \in \Ik} \E \bigg[ \Delta_n \one \bigg\{ U_n \cap \Big\{  \supt h_{r_nt}(\X_\bi) =1, \  \label{e:Dn} \\
&\qquad \qquad \qquad \qquad \qquad \qquad \supt h_{r_nt}(\X_{\bj}) = 0 \text{ for all } \bj\in \Ik \text{ with } \bj \neq \bi \Big\} \bigg\} \bigg]. \notag
\end{align}
Suppose that all the conditions pertaining to the indicator function in \eqref{e:Dn} hold for some $\bi\in \Ik$. Then,
$$
\sum_{(b_i, d_i)\in D_{k,n}, \, b_i \le T} \delta_{(b_i,d_i)} = \delta_{( b(r_n^{-1}\X_\bi), d(r_n^{-1}\X_\bi) )},
$$
and hence, for each $\ell=1,2$,
$$
\sum_{(b_i, d_i)\in D_{k,n}} H_\ell (b_i,d_i) = H_\ell \big(  b(r_n^{-1}\X_\bi), d(r_n^{-1}\X_\bi) \big).
$$
This in turn implies that
\begin{align*}
C_n &= (\seq)^{-1} \sum_{\bi\in \Ik} \E \bigg[ \prod_{\ell=1}^2 \Big( 1-\exp \Big\{ -  \Big(  H_\ell\big( b(r_n^{-1}\X_\bi), d(r_n^{-1}\X_\bi)  \big) -\epsilon_\ell \Big)_+ \Big\} \Big) \\
&\qquad \qquad \quad \times \one \Big\{  U_n \cap \big\{  \supt h_{r_nt}(\X_\bi)=1, \supt h_{r_nt}(\X_{\bj})=0 \text{ for all } \bj\in \Ik \text{ with } \bj \neq \bi \big\}  \Big\} \bigg] \\
&= (\seq)^{-1} \sum_{\bi\in \Ik} \E \bigg[ \prod_{\ell=1}^2 \Big( 1-\exp \Big\{ -  \Big(  H_\ell\big( b(r_n^{-1}\X_\bi), d(r_n^{-1}\X_\bi)  \big) -\epsilon_\ell \Big)_+ \Big\} \Big) \\
&\qquad \qquad \qquad \qquad \qquad \quad \qquad \qquad \qquad \qquad \qquad \qquad\times \one \Big\{  U_n \cap \big\{ \supt h_{r_nt} (\X_\bi)=1  \big\} \Big\}\bigg] \\
&- (\seq)^{-1} \sum_{\bi\in \Ik} \E \bigg[ \prod_{\ell=1}^2 \Big( 1-\exp \Big\{ -  \Big(  H_\ell\big( b(r_n^{-1}\X_\bi), d(r_n^{-1}\X_\bi)  \big) -\epsilon_\ell \Big)_+ \Big\} \Big) \\
&\qquad \qquad \qquad \qquad \qquad \qquad \times \one \Big\{  U_n \cap \bigcup_{\bj\in \Ik, \, \bj\neq \bi} \big\{  \supt h_{r_nt}(\X_\bi)=\supt h_{r_nt}(\X_{\bj})=1\big\}  \Big\} \bigg] \\
&=: G_n-H_n.
\end{align*}
Of the last two terms, $H_n$ is  negligible as $n\to\infty$. Indeed, repeating the same argument as that for proving $D_n\to 0$ can yield that
$$
H_n \le  (\seq)^{-1} \sum_{\ell=0}^{k+1} \sum_{\bi\in \Ik} \sum_{\bj\in \Ik, \, |\bi\cap \bj|=\ell} \P \Big( \supt h_{r_nt} (\X_\bi) = \supt h_{r_nt} (\X_\bj) =1 \Big) \to 0, \ \ \ n\to\infty.
$$
Subsequently, note that, if $\supt h_{r_nt} (\X_\bi)=0$ for some $\bi\in \Ik$, then $b(r_n^{-1} \X_\bi)\ge T$, and therefore, $H_\ell\big( b(r_n^{-1}\X_\bi), d(r_n^{-1}\X_\bi) \big)=0$ for $\ell=1,2$. This implies that
\begin{align*}
G_n &= (\seq)^{-1} \sum_{\bi\in \Ik} \E \bigg[ \prod_{\ell=1}^2 \Big( 1-\exp \Big\{ -  \Big(  H_\ell\big( b(r_n^{-1}\X_\bi), d(r_n^{-1}\X_\bi)  \big) -\epsilon_\ell \Big)_+ \Big\} \Big) \one_{U_n} \bigg] \\
&= (\seq)^{-1} \sum_{\bi\in \Ik} \E \bigg[ \prod_{\ell=1}^2 \Big( 1-\exp \Big\{ -  \Big(  H_\ell\big( b(r_n^{-1}\X_\bi), d(r_n^{-1}\X_\bi)  \big) -\epsilon_\ell \Big)_+ \Big\} \Big) \bigg] \\
&-(\seq)^{-1} \sum_{\bi\in \Ik} \E \bigg[ \prod_{\ell=1}^2 \Big( 1-\exp \Big\{ -  \Big(  H_\ell\big( b(r_n^{-1}\X_\bi), d(r_n^{-1}\X_\bi)  \big) -\epsilon_\ell \Big)_+ \Big\} \Big) \one_{U_n^c} \bigg] \\
&=:I_n -J_n.
\end{align*}
Then,
\begin{align*}
J_n &\le (\seq)^{-1} \sum_{\bi\in\Ik} \E \Big[  \one \big\{  H_1\big( b(r_n^{-1}\X_\bi), d(r_n^{-1}\X_\bi) \big) \ge \epsilon_1 \big\} \times \one_{U_n^c}\Big]  \\
&\le (\seq)^{-1} \sum_{\bi\in\Ik} \E \Big[  \supt h_{r_nt} (\X_\bi)\, \one_{U_n^c} \Big] \\
&\le  (\seq)^{-1} \sum_{\ell=0}^{k+2} \sum_{\bi\in\Ik} \sum_{\bj \in \I_{k+3}, \, |\bi\cap \bj| = \ell}  \E \Big[ \supt h_{r_nt} (\X_\bi)\, \one \big\{  \C(\X_\bj, r_nT) \text{ is connected} \big\} \Big] \\
&= (\seq)^{-1}  \binom{n}{k+2}\binom{n-k-2}{k+3} \E \Big[ \supt h_{r_nt} (X_1,\dots,X_{k+2}) \Big]\\
&\qquad \qquad \qquad \qquad \qquad \qquad \qquad \qquad \qquad \qquad \times \P\Big( \C\big(\{ X_1,\dots,X_{k+3} \}, r_nT\big) \text{ is connected} \Big) \\
&+ (\seq)^{-1} \sum_{\ell=1}^{k+2}\binom{n}{k+2} \binom{k+2}{\ell} \binom{n-k-2}{k+3-\ell} \E \Big[  \supt h_{r_nt} (X_1,\dots,X_{k+2})  \\
&\qquad \qquad \qquad \qquad \qquad \qquad \qquad \times \one \Big\{ \C\big(\{ X_1,\dots,X_\ell, X_{k+3}, \dots, X_{2k+5-\ell} \}, r_nT\big) \text{ is connected}  \Big\} \Big] \\
&\le (\seq)^{-1} n^{2k+5} \int_{(\bbr^d)^{k+2}} \supt h_{r_nt} (x_1,\dots,x_{k+2}) \prod_{i=1}^{k+2} f(x_i) d\bx \\
&\qquad \qquad \qquad\qquad \qquad \times \int_{(\bbr^d)^{k+3}} \one \Big\{ \C \big( \{ x_1,\dots,x_{k+3} \}, r_nT \big) \text{ is connected} \Big\} \prod_{i=1}^{k+3}f(x_i) d\bx \\
&+ \sum_{\ell=1}^{k+2} (\seq)^{-1} n^{2k+5-\ell} \int_{(\bbr^d)^{2k+5-\ell}} \supt h_{r_nt} (x_1,\dots,x_{k+2}) \\
&\qquad \qquad \qquad \qquad  \times \one \Big\{ \C\big(\{ x_1,\dots,x_\ell,x_{k+3}, \dots, x_{2k+5-\ell} \}, r_nT\big) \text{ is connected}  \Big\} \prod_{i=1}^{2k+5-\ell}f(x_i) d\bx \\
&=: K_n + L_n.
\end{align*}
By the customary change of variables as before, we have as $n\to\infty$
$$
K_n \le \|f\|_\infty^{2k+3} n^{k+3} r_n^{d(k+2)} \int_{(\bbr^d)^{k+1}} \supt h_t (0,\by)d\by \int_{(\bbr^d)^{k+2}} \one \Big\{ \C\big( \{ 0,\by \}, T\big) \text{ is connected} \Big\} d\by \to 0,
$$
and
\begin{align*}
L_n &\le \sum_{\ell=1}^{k+2}\| f\|_\infty^{2k+4-\ell} (nr_n^d)^{k+3-\ell} \int_{(\bbr^d)^{2k+4-\ell}} \supt h_t (0,y_1,\dots,y_{k+1}) \\
&\qquad \qquad \qquad \qquad \qquad \times \one \Big\{ \C \big( \{ 0,y_1,\dots,y_{\ell-1}, y_{k+2}, \dots,y_{2k+4-\ell} \}, T \big) \text{ is connected} \Big\}d\by\to 0.
\end{align*}
Thus, it follows that $J_n \to 0$ as $n\to\infty$.

From all of the arguments thus far, we conclude that, as $n\to\infty$,
\begin{align*}
&\eta_n (F_{H_1,H_2,\epsilon_1,\epsilon_2}) = I_n +o(1) \\
&= (\seq)^{-1} \binom{n}{k+2} \int_{(\bbr^d)^{k+2}} \prod_{\ell=1}^2 \bigg( 1-\exp \Big\{  -\Big(H_\ell \big( b( r_n^{-1} \bx), d( r_n^{-1} \bx ) \big) -\epsilon_\ell \Big)_+\Big\}  \bigg) \\
&\qquad \qquad \qquad \qquad \qquad \qquad \qquad\qquad \qquad  \qquad \qquad \qquad\qquad \qquad \times \prod_{i=1}^{k+2} f(x_i) d\bx +o(1) \\
&=(\seq)^{-1} \binom{n}{k+2}  r_n^{d(k+1)} \int_{\bbr^d} \int_{(\bbr^d)^{k+1}} \prod_{\ell=1}^2 \bigg( 1-\exp \Big\{  -\Big(H_\ell \big( b( 0,\by), d(0,\by) \big) -\epsilon_\ell \Big)_+\Big\}  \bigg) \\
&\qquad \qquad \qquad \qquad \qquad \qquad\qquad \qquad \qquad \qquad \qquad\qquad \times f(x)\prod_{i=1}^{k+1}f(x+r_ny_i) d\by dx + o(1) \\
&\to C_k \int_{(\bbr^d)^{k+1}} \prod_{\ell=1}^2 \bigg( 1-\exp \Big\{ -\Big( H_\ell \big( b( 0,\by), d(0,\by) \big) -\epsilon_\ell\Big)_+ \Big\}  \bigg) d\by = \eta (F_{H_1,H_2,\epsilon_1,\epsilon_2}),
\end{align*}
where the last convergence is due to the continuity of $f$ and the dominated convergence theorem.

\section{Appendix}
In this section, the technical results necessary for the proof of Theorem \ref{t:divergence.regime} are collected. As in Section \ref{sec:proof.divergence.regime}, $C>0$ denotes a generic positive constant, independent of $n$.

One of the key ingredients for our proof is the concentration bound for a Poisson $U$-statistics (\cite{bachmann:reitzner:2018}). Let us first rephrase the setup and assumptions given in \cite{bachmann:reitzner:2018} in a way suitable for the current study.
Let $\mathcal P$ be a Poisson point process on $\bbr^d$ with finite intensity measure and $s:(\bbr^d)^i\to \{0,1\}$ be a symmetric indicator function that satisfies the following. \\
$(i)$ There exists $c_1>0$ such that
\begin{equation}  \label{e:cond.s1}
s(x_1,\dots,x_i) = 1 \ \ \text{whenever diam}(x_1,\dots,x_i) < c_1, 
\end{equation}
where diam$(x_1,\dots,x_i)=\max_{1 \le p, q \le i}\| x_p-x_q \|$. \\
$(ii)$ There exists $c_2 > c_1$ such that
\begin{equation}  \label{e:cond.s2}
s(x_1,\dots,x_i) = 0 \ \ \text{whenever diam}(x_1,\dots,x_i) > c_2.
\end{equation}
We then define a \emph{Poisson $U$-statistics} or order $i$ by
\begin{equation}  \label{e:Poisson.U.stats}
F(\mathcal P) := \sum_{\Y\subset \mathcal P, \, |\Y|=i} s(\Y).
\end{equation}

\begin{proposition} [Theorem 3.1 in \cite{bachmann:reitzner:2018}]  \label{p:concentration}
There exists a constant $C>0$, depending only on $i$, $d$, $c_1$, and $c_2$, such that, for all $r>0$,
\begin{equation}  \label{e:concentration1}
\P\Big( F(\mathcal P)\ge \E\big[F(\mathcal P)\big] +r\Big) \le \exp \Big\{ -C \Big[ \big( \E[F(\mathcal P)] +r \big)^{1/(2i)} - \big( \E[F(\mathcal P)] \big)^{1/(2i)}  \Big]^2 \Big\}
\end{equation}
and
\begin{equation}  \label{e:concentration2}
\P\Big( F(\mathcal P)\le \E\big[F(\mathcal P)\big] -r\Big) \le \exp \Big\{ -\frac{Cr^2}{\text{Var}\big(F(\mathcal P) \big)} \Big\}.
\end{equation}
\end{proposition}

For the next lemma, recall the sequences $(p_m)$ and $(q_m)$ in \eqref{e:def.pm.qm}, and $v_m=\lfloor e^{m^\gamma} \rfloor$ for some $\gamma\in (0,1)$.

\begin{lemma} \label{l:RV.seq}
Let $u_m$ and $w_m$ be any of the sequences $p_m$, $q_m$, $v_m$, and $v_{m+1}$. Assume that $(r_n)$ is a regularly varying sequence with exponent $\rho<0$. Then,
$$
\frac{u_m}{w_m} \to 1, \ \  \ m\to\infty,
$$
and
$$
\frac{r_{u_m}}{r_{w_m}} \to 1, \ \  \ m\to\infty.
$$
\end{lemma}

\begin{proof}
By the definition of these sequences, it is clear that
$$
\frac{v_m}{v_{m+1}} \le \frac{u_m}{w_m} \le \frac{v_{m+1}}{v_m}.
$$
We see that
$$
1 \le \frac{v_{m+1}}{v_m} \le \frac{e^{(m+1)^\gamma-m^\gamma}}{1-e^{-m^\gamma}} = \frac{e^{m^{\gamma-1} (\gamma + o(1))}}{1-e^{-m^\gamma}}, \ \ \ m\to\infty.
$$
As $0<\gamma<1$, the rightmost term goes to $1$ as $m\to\infty$.

For the proof of the second statement, we write
$$
\frac{r_{u_m}}{r_{w_m}} = \frac{r_{\lfloor w_m \frac{u_m}{w_m} \rfloor}}{r_{w_m}} - \Big( \frac{u_m}{w_m} \Big)^\rho  + \Big( \frac{u_m}{w_m} \Big)^\rho.
$$
As shown above, $u_m/w_m\to 1$ as $m\to\infty$, and therefore, $u_m/w_m\in [\frac{1}{2}, \frac{3}{2}]$ for sufficiently large $m$. Hence, as $m\to\infty$,
$$
\Big|  \, \frac{r_{\lfloor w_m \frac{u_m}{w_m} \rfloor}}{r_{w_m}} - \Big( \frac{u_m}{w_m} \Big)^\rho  \, \Big| \le \sup_{\frac{1}{2}\le a \le \frac{3}{2}} \Big|  \, \frac{r_{\lfloor w_m a \rfloor}}{r_{w_m}} - a^\rho  \, \Big| \to 0,
$$
by the uniform convergence of regularly varying sequences (see Proposition 2.4 in \cite{resnick:2007}).
\end{proof}

The following lemma gives the asymptotic moments of various quantities used for the proof of Theorem \ref{t:divergence.regime}. Recall the notations \eqref{e:def.Akst}, \eqref{e:def.Tm.uparrow}, \eqref{e:def.Tm.downarrow}, and \eqref{e:def.Vm}.

\begin{lemma} \label{l:moments}
Under the assumptions of Theorem \ref{t:divergence.regime}, we have that, as $m\to\infty$,
\begin{equation}  \label{e:moment.Tm.uparrow}
\E \big[ T_m^\uparrow(s,t) \big] \sim A_k(s,t) v_m^{k+2}r_{q_m}^{d(k+1)},
\end{equation}
\begin{equation}  \label{e:moment.Tm.downarrow}
\E \big[ T_m^\downarrow(s,t) \big] \sim A_k(s,t) v_{m+1}^{k+2}r_{p_m}^{d(k+1)},
\end{equation}
\begin{equation}  \label{e:2nd.moment.Tm.downarrow}
\text{Var}\big( T_m^\downarrow(s,t) \big) \sim A_k(s,t) v_{m+1}^{k+2}r_{p_m}^{d(k+1)},
\end{equation}
\begin{equation}  \label{e:moment.Vm}
\E \big[ V_m(t) \big] \sim \frac{v_{m}^{k+3}r_{q_m}^{d(k+2)}}{(k+3)!}\, \int_{\bbr^d}f(x)^{k+3}dx \int_{(\bbr^d)^{k+2}} \one \big\{ \C\big( \{0,\by\}, t \big) \text{ is connected} \big\}d\by, 
\end{equation}
where $\by=(y_1,\dots,y_{k+2})\in (\bbr^d)^{k+2}$. 
\end{lemma}

\begin{proof}
The proofs of \eqref{e:moment.Tm.uparrow}, \eqref{e:moment.Tm.downarrow}, and \eqref{e:moment.Vm} are very similar, and therefore, we prove the first one only.
By the Palm theory for Poisson point processes (see, e.g., Section 1.7 in \cite{penrose:2003}),
$$
\E \big[ T_m^{\uparrow}(s,t) \big] = \frac{v_{m+1}^{k+2}}{(k+2)!}\, \int_{(\bbr^d)^{k+2}} h_{r_{p_m}s}(x_1,\dots,x_{k+2})h_{r_{p_m}t}(x_1,\dots,x_{k+2})\prod_{i=1}^{k+2} f(x_i)d\bx.
$$
Making the change of variables, $x_i=x+r_{p_m}y_{i-1}$, $i=1,\dots,k+2$ (with $y_0\equiv 0$),
\begin{align*}
\E \big[ T_m^{\uparrow}(s,t) \big] = \frac{v_{m+1}^{k+2}r_{p_m}^{d(k+1)}}{(k+2)!}\, \int_{\bbr^d}\int_{(\bbr^d)^{k+1}} h_s(0,\by)h_t(0,\by) f(x)\prod_{i=1}^{k+1}f(x+r_{p_m}y_i)d\by dx, 
\end{align*}
where $\by=(y_1,\dots,y_{k+1})\in (\bbr^d)^{k+1}$. 
As $f$ is continuous, we have $\prod_{i=1}^{k+1}f(x+r_{p_m}y_i) \to f(x)^{k+1}$ as $m\to\infty$ a.e.
By the dominated convergence theorem and Lemma \ref{l:RV.seq}, we conclude that
$$
\E \big[ T_m^{\uparrow}(s,t) \big]  \sim A_k(s,t) v_{m+1}^{k+2}r_{p_m}^{d(k+1)} \sim A_k(s,t) v_{m}^{k+2}r_{q_m}^{d(k+1)}, \ \ \text{as } m\to\infty.
$$

Finally, in the sparse regime, the variance of Poisson $U$-statistics as in \eqref{e:Poisson.U.stats} is known to exhibit the same growth rate as its expectation. Although we omit a formal proof of \eqref{e:2nd.moment.Tm.downarrow}, those wishing for more detailed arguments may refer to \cite{kahle:meckes:2013} and Proposition 6.2 in \cite{owada:thomas:2020}.
\end{proof}

The objective of the final lemma below is to justify that, under the appropriate scaling, the difference between the Poisson $U$-statistics and the $U$-statistics induced by the corresponding binomial process, asymptotically vanishes.

\begin{lemma}  \label{l:diff.Poisson.binomial}
Under the assumptions of Theorem \ref{t:divergence.regime}, we have that, as $m\to\infty$,
\begin{equation}  \label{e:first.diff.Poisson.binomial}
\big(v_m^{k+2}r_{q_m}^{d(k+1)}\big)^{-1} \Big[  \sum_{\Y\subset \X_{v_{m+1}}, \, |\Y|=k+2} h_{r_{p_m}s}(\Y)h_{r_{p_m}t}(\Y) - T_m^{\uparrow}(s,t) \Big] \to 0, \ \ \text{a.s.},
\end{equation}
$$
\big(v_{m+1}^{k+2}r_{p_m}^{d(k+1)}\big)^{-1} \Big[  \sum_{\Y\subset \X_{v_{m}}, \, |\Y|=k+2} h_{r_{q_m}s}(\Y)h_{r_{q_m}t}(\Y) - T_m^{\downarrow}(s,t) \Big] \to 0, \ \ \text{a.s.},
$$
$$
\big( v_m^{k+2} r_{q_m}^{d(k+1)} \big)^{-1} \Big[\sum_{\substack{\Y \subset \X_{v_{m+1}}, \\ |\Y|=k+3}} \one \big\{ \C(\Y, r_{p_m}t) \text{ is connected} \big\} -V_{m}(t) \Big] \to 0, \ \ \text{a.s.}
$$
\end{lemma}

\begin{proof}
The proofs of these statements are essentially the same; therefore, we prove only \eqref{e:first.diff.Poisson.binomial}. By the Borel--Cantelli lemma, together with Markov's inequality, it suffices to demonstrate that
\begin{equation}  \label{e:BC.diff}
\sum_{m=1}^\infty \frac{1}{v_m^{k+2}r_{q_m}^{d(k+1)}} \E \Big[  \Big| \sum_{\Y\subset \X_{v_{m+1}}, \, |\Y|=k+2} h_{r_{p_m}s}(\Y)h_{r_{p_m}t}(\Y) - T_m^\uparrow(s,t) \Big| \Big]<\infty.
\end{equation}
Recall that $N_{v_{m+1}}=|\mathcal P_{v_{m+1}}|$ is Poisson distributed with mean $v_{m+1}$. By conditioning on the values of $N_{v_{m+1}}$, the expectation portion in \eqref{e:BC.diff} is equal to
\begin{align*}
&\sum_{\ell=0}^\infty \E  \Big[  \Big| \sum_{\Y\subset \X_{v_{m+1}}, \, |\Y|=k+2} h_{r_{p_m}s}(\Y)h_{r_{p_m}t}(\Y) - \sum_{\Y\subset \X_{\ell}, \, |\Y|=k+2} h_{r_{p_m}s}(\Y)h_{r_{p_m}t}(\Y) \Big| \Big] \P(N_{v_{m+1}}=\ell)  \\
&=\sum_{\ell=0}^\infty \Big| \binom{\ell}{k+2} -\binom{v_{m+1}}{k+2}\Big| \E \big[ h_{r_{p_m}s}(X_1,\dots,X_{k+2}) h_{r_{p_m}t}(X_1,\dots,X_{k+2}) \big] \P(N_{v_{m+1}}=\ell),
\end{align*}
where $X_1,\dots,X_{k+2}$ are iid random variables with density $f$.
By the same change of variables as in the proof of Lemma \ref{l:moments}, together with Lemma \ref{l:RV.seq}, it holds that
$$
\E \big[ h_{r_{p_m}s}(X_1,\dots,X_{k+2}) h_{r_{p_m}t}(X_1,\dots,X_{k+2}) \big] \sim Cr_{p_m}^{d(k+1)} \sim Cr_{q_m}^{d(k+1)}, \ \ \ m\to\infty.
$$
Now, the left hand side of \eqref{e:BC.diff} is bounded, up to multiplicative constants, by
\begin{align}
&\sum_{m=1}^\infty \frac{1}{v_{m+1}^{k+2}}\, \sum_{\ell=0}^\infty \Big| \binom{\ell}{k+2} -\binom{v_{m+1}}{k+2}  \Big| \,  \P\big( N_{v_{m+1}}=\ell \big) \label{e:CS.inequ} \\
&=\sum_{m=1}^\infty \frac{1}{v_{m+1}^{k+2}}\, \E \bigg[ \Big| \binom{N_{v_{m+1}}}{k+2} -\binom{v_{m+1}}{k+2}  \Big|  \bigg] \notag \\
&\le \sum_{m=1}^\infty \frac{1}{v_{m+1}^{k+2}}\, \bigg\{ \E \bigg[  \binom{N_{v_{m+1}}}{k+2}^2 \bigg] - 2\binom{v_{m+1}}{k+2} \E \bigg[  \binom{N_{v_{m+1}}}{k+2} \bigg] + \binom{v_{m+1}}{k+2}^2  \bigg\}^{1/2}, \notag
\end{align}
where the last line is due to the Cauchy--Schwarz inequality. It is then elementary to check that
$$
\E \bigg[  \binom{N_{v_{m+1}}}{k+2} \bigg] = \frac{v_{m+1}^{k+2}}{(k+2)!},  \ \ \ \ \  \E \bigg[  \binom{N_{v_{m+1}}}{k+2}^2 \bigg] = \frac{1}{\big( (k+2)! \big)^2}\, \sum_{j=1}^{2k+4} c_j v_{m+1}^j
$$
for some $c_j$'s with $c_{2k+4}=1$. Therefore, the last expression in \eqref{e:CS.inequ} is equal to
\begin{equation}  \label{e:before.final}
\sum_{m=1}^\infty \frac{1}{v_{m+1}^{k+2}}\, \bigg\{ \frac{1}{\big( (k+2)! \big)^2}\, \sum_{j=1}^{2k+3} c_j^\prime v_{m+1}^j \bigg\}^{1/2}
\end{equation}
for some $c_j^\prime$'s, $j=1,\dots,2k+3$ (note that $v_{m+1}^{2k+4}$ has been canceled here). Finally, \eqref{e:before.final} is further bounded by
$$
C\sum_{m=1}^\infty \frac{1}{v_{m+1}^{1/2}} \le C\sum_{m=1}^\infty e^{-m^\gamma/2} < \infty,
$$
and the proof of \eqref{e:BC.diff} is complete.
\end{proof}

\noindent \textbf{Acknowledgements}: 
The author is very grateful for 
useful comments received from two anonymous referees and an anonymous
Associate Editor. These comments helped the author to introduce a number of improvements
to the paper.

%\bibliography{Takashi_ref}

\end{document}